\documentclass[11pt,letterpaper,reqno]{amsart}

\author{Tomoki Ohsawa}
\address{Department of Mathematics \& Statistics, University of Michigan--Dearborn, 4901 Evergreen Road, Dearborn, MI 48128-2406}
\email{ohsawa@umich.edu}
\title[Symmetry Reduction of Optimal Control Systems]{Symmetry Reduction of Optimal Control Systems and Principal Connections}
\date{\today}

\keywords{optimal control, symmetry and reduction, momentum maps, principal connections, Hamiltonian reduction, Poisson reduction}

\subjclass[2010]{49J15, 53D20, 37J15, 70H05, 70H25}

\usepackage{graphicx, minibox}
\usepackage[margin=1in, marginpar=.5in]{geometry}

\usepackage[all]{xy}

\usepackage[numbers,sort&compress]{natbib}

\usepackage[colorlinks=false]{hyperref}

\theoremstyle{plain}
\newtheorem{theorem}{Theorem}[section]

\newtheorem{lemma}[theorem]{Lemma}
\newtheorem{proposition}[theorem]{Proposition}

\theoremstyle{definition}
\newtheorem{definition}[theorem]{Definition}
\newtheorem{example}[theorem]{Example}
\newtheorem{problem}[theorem]{Problem}

\theoremstyle{remark}
\newtheorem{remark}[theorem]{Remark}

\def\od#1#2{\dfrac{d#1}{d#2}}
\def\pd#1#2{\dfrac{\partial #1}{\partial #2}}

\def\covd#1#2{\dfrac{D#1}{D#2}}
\def\tcovd#1#2{D#1/D#2}

\def\parentheses#1{\!\left(#1\right)}
\def\brackets#1{\!\left[#1\right]}
\def\braces#1{\!\left\{#1\right\}}

\def\Span{\mathop{\mathrm{span}}\nolimits} 

\def\DS{\displaystyle}
\def\R{\mathbb{R}}
\def\defeq{\mathrel{\mathop:}=}
\def\eqdef{=\mathrel{\mathop:}}
\def\setdef#1#2{ \left\{ #1 \ |\ #2 \right\} }
\def\ip#1#2{\left\langle#1,#2\right\rangle}

\newcommand{\hor}{\operatorname{hor}}

\newcommand{\hl}{\operatorname{hl}}

\def\eps{\varepsilon}

\makeatletter
\def\blfootnote{\xdef\@thefnmark{}\@footnotetext}
\makeatother

\def\FL{\mathbb{F}L}
\def\FH{\mathbb{F}H}
\def\F{\mathbb{F}}
\newcommand{\Ad}{\operatorname{Ad}}
\newcommand{\ad}{\operatorname{ad}}

\begin{document}

\footskip=.6in

\blfootnote{{\em Note}: This is a modified and expanded version of the conference paper~\cite{Oh2011}.}

\begin{abstract}
  This paper explores the role of symmetries and reduction in nonlinear control and optimal control systems.
  The focus of the paper is to give a geometric framework of symmetry reduction of optimal control systems as well as to show how to obtain explicit expressions of the reduced system by exploiting the geometry.
  In particular, we show how to obtain a principal connection to be used in the reduction for various choices of symmetry groups, as opposed to assuming such a principal connection is given or choosing a particular symmetry group to simplify the setting.
  Our result synthesizes some previous works on symmetry reduction of nonlinear control and optimal control systems.
  Affine and kinematic optimal control systems are of particular interest: We explicitly work out the details for such systems and also show a few examples of symmetry reduction of kinematic optimal control problems.
\end{abstract}

\maketitle

\section{Introduction}
\subsection{Background}
Many control systems, particularly those arising from mechanical systems, have symmetries---often translational and rotational, and sometimes combinations of them.
Such a symmetry is usually described as an invariance or equivariance under an action of a Lie group, and the system can be reduced to a lower-dimensional one or decoupled into subsystems by exploiting the symmetry.
\citet{NiSc1982} and \citet{GrMa1985} formulated symmetries of nonlinear control systems from the differential-geometric point of view, and also showed how one can reduce a control system with symmetry to a quotient space.

Likewise, optimal control systems also have such symmetries.
\citet{GrMa1984} showed that, in relation to the work in \cite{GrMa1985}, one can decompose optimal feedback laws by exploiting the symmetries of control systems; \citet{Sc1987} showed a method to analyze symmetries of optimal Hamiltonians without explicitly calculating them, while \citet{LeCoDiMa2004} analyzed symmetries of vakonomic systems and applied their result to optimal control problems, \citet{EcMaMuRo2003} from the pre-symplectic point of view, and \citet{BlSc2000} and \citet{IbPeSa2010} using Dirac structures.

Symmetry reduction of optimal control systems are desirable from a computational point of view as well.
Given that solving optimal control problems usually involves iterative methods such as the shooting method (as opposed to solving a single initial value problem), reducing the system to a lower-dimensional one results in a considerable reduction of the computational cost.

From a theoretical point of view, a certain class of optimal control problems has a rich geometric structure, and provides many interesting questions relating differential-geometric ideas with control-theoretic problems.
Most notably, \citet{Mo1990, Mo1991a, Mo1993a, Mo1993b, Mo2002}, following the work of \citet{ShWi1987, ShWi1989}, explored optimal control of deformable bodies, such as the falling cat problem, from the differential-geometric point of view.
In particular, principal bundles, along with principal connections on them defined by momentum maps, are identified as a natural geometric setting for such problems.
The same geometric setting applies to kinematic control of nonholonomic mechanical systems (see, e.g., \citet{KeMu1995}, \citet[Chapters~7 and 8]{MuLiSa1994}, and \citet{LiCa1993}), where the principal connections are defined by the constraints instead of momentum maps.
This geometric setting also gives rise to geometric phases and holonomy (see, e.g., \citet{MaMoRa1990} and references therein), which have applications in motion generation of mechanical systems by shape change.

\subsection{Main Results and Comparison with Existing Literature}
\label{ssec:MainResultsAndComparison}
Figure~\ref{fig:RelatingOCPs} gives a schematic overview of the results in the paper and their relationships.
\begin{figure}[h!]
  \sf
  \centerline{
    \xymatrix@!0@R=1.1in@C=2in{
       \ar[d]_{\text{(Section~\ref{sec:SymmetryInNonlinearControlSystems})}}^{\text{Reduction}}="a"
       \framebox{\small\strut\minibox[c]{Control system\\with symmetry}} \ar[r]^{\text{Cost function}}_{\text{with symmetry}}  & \framebox{\small\strut\minibox[c]{Optimal control\\problem}} \ar[r]^{\text{PMP}\quad} & \framebox{\small\strut\minibox[c]{Hamiltonian system\\with symmetry}} \ar[d]^{\text{(Section~\ref{sec:SymRedInOptimalControlSystems})}}_{\text{Reduction}}="b" \POS(51,-14.2)*+\txt{\raisebox{-2.5pc}{\minibox[c]{\footnotesize Principal connection\\\scriptsize(Section~\ref{sec:PrincipalConnection})}}} \ar@{=>} "a" \ar@{=>} "b"
      \\
      \framebox{\small\strut\minibox[c]{Reduced\\control system}} \ar@{-->}[rr]_{\text{Reduced PMP (Section~\ref{ssec:PoissonRedOfPMP})}} & & \framebox{\small\strut\minibox[c]{Reduced\\Hamiltonian system}}
    }
  }
  \caption{Schematic overview of reduction of control and optimal control systems with references to corresponding sections in the paper. See also the outline in Section~\ref{ssec:Outline}.}
  \label{fig:RelatingOCPs}
\end{figure}
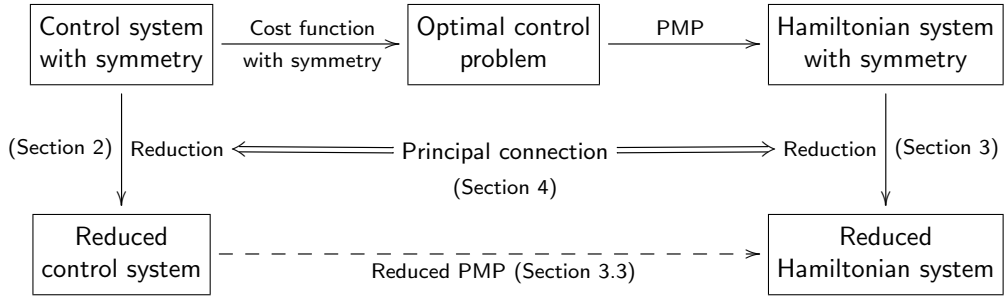
We first characterize symmetries in nonlinear control systems and use a principal connection to reduce such systems.
We then discuss the associated symmetries in optimal control problem of such systems following \citet{GrMa1984}, and apply Hamiltonian reduction theory to the Pontryagin maximum principle (PMP) for optimal control systems with symmetries; the principal connection plays an important role here as well.
In particular, we apply the Poisson reduction of \citet{CeMaPeRa2003} to the Hamiltonian system given as a necessary condition for optimality by the Pontryagin maximum principle.
The resulting Hamilton--Poincar\'e equations give a reduced set of equations for optimality, and are naturally considered as a reduced maximum principle applied to the reduced control system.

We note that \citet{IbPeSa2010} study a similar problem in a more general setting with a slightly different focus: We assume that one can eliminate the control to obtain a Hamiltonian system on a cotangent bundle, whereas \citet{IbPeSa2010} do not make the assumption and exploit Dirac structures to handle those cases where one cannot easily obtain a Hamiltonian system explicitly, an approach originally due to \citet{BlSc2000}.
Therefore, our problem setting and geometric framework are in fact a special case of those in \cite{IbPeSa2010}.
At the expense of generality, however, we focus on the practical issue of obtaining an explicit expression for the reduced system.
Specifically, {\em our specialization leads us to a prescription to obtain a principal connection for a given optimal control system with various symmetries, as opposed to assuming, as in \cite{IbPeSa2010}, that it is given at the outset or deliberately choosing a particular symmetry group to simplify the geometric setting}.
In particular, for affine and kinematic optimal control systems, we may explicitly characterize the principal connection using the nonholonomic connection of \citet{BlKrMaMu1996}.

We also note that a reduced maximum principle (see Fig.~\ref{fig:RelatingOCPs} and Section~\ref{ssec:PoissonRedOfPMP}) is discussed by \citet{BlSc2000}; however their definition of symmetry of control systems is slightly more restrictive compared to that of \cite{IbPeSa2010} and the present paper (see Remark~\ref{remark:ComparisonWithBlSc2000}).
Note also that, in their setup, it is shown that their reduced equations are simplified due to the transversality condition.
However, this is not in general true for the case with fixed endpoints and thus the reduced equations become more complicated (see Remark~\ref{remark:ComparisonWithBlSc2000-2}).

{\em The construction of principal connections developed here turns out to be a generalization of the mechanical connection used in the falling cat problem as well as those used in kinematic control of nonholonomic systems}.
In the falling cat problem, there is a natural choice of principal connection that arises from the problem setting, but the same construction of principal connection applies to kinematic control problems only by choosing certain symmetry groups to realize the same geometric setting; in other words, one does not have the freedom to choose the symmetry group to be used in the reduction.
{\em Our construction does not have such restriction and hence can be applied to a wider class of control systems with symmetries}.

As a result, we synthesize some previous works by showing how the basic settings of those works arise as special cases of our result; these include optimal control of deformable bodies mentioned above and also the Lie--Poisson reduction of optimal control systems on Lie groups of \citet{Kr1993}.

\subsection{Outline}
\label{ssec:Outline}
We first define, in Section~\ref{sec:SymmetryInNonlinearControlSystems}, symmetries in nonlinear control systems, and show reduction of such systems by the symmetries (see Fig.~\ref{fig:RelatingOCPs}).
Section~\ref{sec:SymRedInOptimalControlSystems} first briefly discusses Poisson reduction of \citet{CeMaPeRa2003} for Hamiltonian systems and then applies it to the Hamiltonian system defined by the maximum principle and obtain the Hamilton--Poincar\'e equations for such systems.
Section~\ref{sec:PrincipalConnection} addresses the issue of how one should choose the principal connection to be used in the reduction of optimal control systems.
Section~\ref{sec:Examples} gives various examples to show how the theory specializes to several previous works on the subject as well as to illustrate how the reduction decouples the optimal control system.

\section{Symmetry and Reduction of Nonlinear Control Systems}
\label{sec:SymmetryInNonlinearControlSystems}
\subsection{Nonlinear Control Systems}
Let $M$ be a smooth manifold and $\tau_{M}: TM \to M$ be its tangent bundle; let $E \defeq M \times \R^{d}$ and see $\pi^{E}: E \to M$ as a (trivial) vector bundle\footnote{More generally, we may take a fiber bundle for $E$ (see, e.g., \citet{NiSc1982} and references therein).}; also let $f: E \to TM$ be a fiber-preserving smooth map, i.e., the diagram
\begin{equation*}
  \xymatrix@R=0.225in@C=.225in{
    E \ar[rr]^{f} \ar[rd]_{\pi^{E}\!} & & TM \ar[ld]^{\!\tau_{M}}
    \\
    & M &
  }
\end{equation*}
commutes.
Then a {\em nonlinear control system} is defined by
\begin{equation}
  \label{eq:NonlinearControlSystem}
  \dot{x} = f(x, u).
\end{equation}

\subsection{Symmetry in Nonlinear Control Systems}
\label{ssec:SymmetryInNonlinearControlSystems}
Following \citet{Sc1981} and \citet{NiSc1982} (see also \citet{GrMa1985}  and \citet{Sc1987}), we assume that the control system~\eqref{eq:NonlinearControlSystem} has a symmetry in the following sense:
Let $G$ be a free and proper (left) Lie group action on $M$.
We have $\Phi: G \times M \to M$ or $\Phi_{g}: M \to M$ for any $g \in G$; as a result we have the principal bundle
\begin{equation*}
  \pi: M \to M/G.
\end{equation*}
The action $\Phi_{g}$ gives rise to the tangent lift $T\Phi_{g}: TM \to TM$.
Let us also assume that we have a linear representation of $G$ on the control space $\R^{d}$, i.e., we have a representation $\sigma_{(\cdot)}: G \to GL(d, \R)$.
Then we define an action of $G$ on $E = M \times \R^{d}$ as follows:
\begin{equation}
  \label{eq:Psi}
  \Psi_{g}: E \to E;
  \quad
  (x, u) \mapsto \parentheses{ \Phi_{g}(x), \sigma_{g}(u) } = (g x, g u),
\end{equation}
where we introduced the shorthand notation $g x \defeq \Phi_{g}(x)$ and $g u \defeq \sigma_{g}(u)$.
\begin{remark}
  \label{remark:NontrivialActionToControls}
  In many examples, the representation $\sigma_{(\cdot)}: G \to GL(d, \R)$ turns out to be trivial.
  However, there are non-trivial cases as well: See Section~\ref{ssec:ClebschOptimalControl}.
\end{remark}

We are now ready to define a symmetry for a nonlinear control system (see \citet[Definition~8]{IbPeSa2010}): We say that the nonlinear control system~\eqref{eq:NonlinearControlSystem} has a $G$-symmetry if the map $f: E \to TM$ is equivariant under the $G$-actions on $E$ and $TM$ defined above, i.e.,
\begin{equation}
  \label{eq:f-symmetry}
  T\Phi_{g} \circ f = f \circ \Psi_{g},
\end{equation}
or the diagram
\begin{equation*}
  \xymatrix@!0@R=0.55in@C=.7in{
    E \ar[r]^{f} \ar[d]_{\Psi_{g}} & TM \ar[d]^{T\Phi_{g}}
    \\
    E \ar[r]_{f} & TM 
  }
\end{equation*}
commutes for any $g \in G$.

\begin{remark}
  \label{remark:ComparisonWithBlSc2000}
  Note that this definition of symmetry is more general compared to that of \citet{BlSc2000}.
  Their definition of symmetry of the control vector fields \cite[Eq.~(14)]{BlSc2000}, i.e.,
  \begin{equation}
    \label{eq:SymmetryInBlSc2000}
    [f(\cdot,u), \mathfrak{g}_{M}] = 0,
  \end{equation}
  where $\mathfrak{g}_{M}$ is the set of infinitesimal generators, is rather restrictive for us, since this is not true in general if $f(\cdot,u)$ has vertical components and $G$ is non-Abelian.
  To illustrate it, consider the extreme case where $M = G$, a non-abelian Lie group.
  Then the system is a left-invariant control system on $G$ (see Section~\ref{ssec:ControlSystemsOnLieGroups}); but then $\mathfrak{g}_{M} = TG$ and $f(x,u) \in TG$ and so Eq.~\eqref{eq:SymmetryInBlSc2000} does not hold except for the very special case where $f(\cdot,u)$ commutes with every vector field on $G$.
\end{remark}

\subsection{Symmetry in affine control systems}
\label{ssec:SymmetryInAffineControlSystems}
Consider an {\em affine control system}, i.e., Eq.~\eqref{eq:NonlinearControlSystem} with
\begin{equation}
  \label{eq:f-affine}
  f(x, u) = X_{0}(x) + \sum_{i=1}^{d} u_{i} X_{i}(x),
\end{equation}
where the control vector fields $\{ X_{i} \}_{i=1}^{d}$ are linearly independent on $M$.
Let $\mathcal{D} \subset TM$ be the distribution defined by
\begin{equation}
  \label{eq:mathcalD}
  \mathcal{D} \defeq \Span\{X_{1}, \dots, X_{d}\}.
\end{equation}
We assume that the vector field $X_{0}$ is $G$-invariant, i.e., for any $g \in G$,
\begin{equation}
  \label{eq:X_0-symmetry}
  T\Phi_{g} \circ X_{0} = X_{0} \circ \Phi_{g},
\end{equation}
and also that the distribution is invariant under the tangent lift of the $G$-action on $Q$, i.e.,
\begin{equation}
  \label{eq:mathcalD-symmetry}
  T\Phi_{g}(\mathcal{D}) = \mathcal{D}
\end{equation}
for any $g \in G$.  This implies that, for each vector field $X_{i}$ for $i = 1, \dots, d$ and any $x \in M$ and $g \in G$, we have
\begin{equation}
  \label{eq:R}
  T_{x}\Phi_{g}\parentheses{ X_{i}(x) } = \sum_{j=1}^{d} R_{i}^{j}(x,g)\,X_{j}(g x),
\end{equation}
where $R(x,g)$ is an invertible $d \times d$ matrix.
This gives rise to an action of $G$ on $E = M \times \R^{d}$, i.e., $\Psi_{g}: E \to E$ defined by
\begin{equation*}
  \Psi_{g}: (x, u) \mapsto \parentheses{ g x, R^{T}(x,g) u }.
\end{equation*}
Then the $G$-symmetry of $X_{0}$ and $\mathcal{D}$, i.e., Eqs.~\eqref{eq:X_0-symmetry} and \eqref{eq:mathcalD-symmetry}, implies that of $f$, i.e.,
\begin{equation*}
  T_{x}\Phi_{g}(f(x,u)) = f \circ \Psi_{g}(x, u).
\end{equation*}
In particular, consider the case where $R(x,g)$ has no dependence on $x$, i.e., $R(x,g) = R(g)$; this is the case if, for example, $M$ is a vector space and the action $\Phi_{g}: M \to M$ is linear.
Then the matrix $R^{T}(g)$ gives the representation $\sigma_{(\cdot)}: G \to GL(d, \R)$, i.e., $\sigma_{g} = R^{T}(g)$.

\subsection{Reduced Control System}
The equivariance of the map $f$ shown above gives rise to the map $\bar{f}: E/G \to TM/G$ defined so that the diagram
\begin{equation*}
  \xymatrix@!0@R=0.55in@C=.8in{
    E \ar[r]^{f} \ar[d]_{\pi^{E}_{G}} & TM \ar[d]^{\pi^{TM}_{G}}
    \\
    E/G \ar[r]_{\bar{f}} & TM/G
  }
\end{equation*}
commutes, where $\pi^{E}_{G}: E \to E/G$ and $\pi^{TM}_{G}: TM \to TM/G$ are both quotient maps.
Then the map $\bar{f}$ defines the reduced control system.

Since $E = M \times \R^{d}$, the quotient $E/G$ defines the associated bundle
\begin{equation*}
  E/G = (M \times \R^{d})/G = M \times_{G} \R^{d},
\end{equation*}
which is a vector bundle over $M/G$ (see, e.g., \citet[Section~2.3]{CeMaRa2001}).
On the other hand, again following \citet[Section~2.3]{CeMaRa2001}, the quotient $TM/G$ is identified with $T(M/G) \oplus \tilde{\mathfrak{g}}$, where $\tilde{\mathfrak{g}}$ is the associated bundle defined as
\begin{equation*}
  \tilde{\mathfrak{g}} \defeq M \times_{G} \mathfrak{g} = (M \times \mathfrak{g})/G
\end{equation*}
with $\mathfrak{g}$ being the Lie algebra of the Lie group $G$.
More specifically, given a principal bundle connection form
\begin{equation}
  \label{eq:mathcalA}
  \mathcal{A}: TM \to \mathfrak{g},
\end{equation}
we have the identification~(see \cite[Section~2.4]{CeMaRa2001})
\begin{equation}
  \label{eq:alpha_mathcalA}
  \alpha_{\mathcal{A}}: TM/G \to T(M/G) \oplus \tilde{\mathfrak{g}};
  \qquad
  [v_{x}]_{G} \mapsto T_{x}\pi(v_{x}) \oplus [ x, \mathcal{A}_{x}(v_{x}) ]_{G},
\end{equation}
where $[\,\cdot\,]_{G}$ stands for an equivalence class defined by the $G$-action.
Therefore, we may introduce the maps $\bar{f}_{M/G}: E/G \to T(M/G)$ and $\bar{f}_{\tilde{\mathfrak{g}}}: E/G \to \tilde{\mathfrak{g}}$ defined by
\begin{equation*}
  \bar{f}_{M/G}([x, u]_{G}) \defeq T_{x}\pi \circ f(x, u),
  \qquad
  \bar{f}_{\tilde{\mathfrak{g}}}([x, u]_{G}) \defeq [x, \mathcal{A}_{x}(f(x,u)) ]_{G}
\end{equation*}
for any element $[x, u]_{G} \in E/G = M \times_{G} \R^{d}$; these maps are clearly well-defined because of the equivariance of $f$.
Then we have
\begin{equation*}
  \alpha_{\mathcal{A}} \circ \bar{f} = \bar{f}_{M/G} \oplus \bar{f}_{\tilde{\mathfrak{g}}},
\end{equation*}
and thus the reduced system is decoupled into two subsystems:
\begin{equation}
  \label{eq:ReducedControlSystem}
  \dot{\bar{x}} = \bar{f}_{M/G}(\bar{u}_{\bar{x}}),
  \qquad
  \tilde{\xi}_{\bar{x}} = \bar{f}_{\tilde{\mathfrak{g}}}(\bar{u}_{\bar{x}}),
\end{equation}
where $\bar{x} \defeq \pi(x)$, $\bar{u}_{\bar{x}} \defeq [x, u]_{G}$, and $\tilde{\xi}_{\bar{x}} \defeq [x, \mathcal{A}_{x}(\dot{x}) ]_{G}$.

\section{Symmetry and Reduction of Optimal Control Systems}
\label{sec:SymRedInOptimalControlSystems}
This section first summarizes the fact that the $G$-symmetry of a nonlinear control system implies that of the corresponding optimal control system if the cost function is also $G$-invariant.
We note that similar results are briefly discussed in \citet{GrMa1984}.
We then show how a Poisson reduction may be applied to reduce the optimal control system with symmetry.

\subsection{Pontryagin Maximum Principle and Symmetry in Optimal Control}
\label{ssec:PMPandSymmetryInOptimalControl}
Given a cost function $C: E \to \R$ and fixed times $t_{0}$ and $t_{1}$ such that $t_{0} < t_{1}$, define the cost functional
\begin{equation*}
  J \defeq \int_{t_{0}}^{t_{1}} C(x(t), u(t))\,dt.
\end{equation*}
Let $x_{0}$ and $x_{1}$ be fixed in $M$.
Then we formulate an {\em optimal control problem} as follows:
Minimize the cost functional, i.e., 
\begin{equation*}
  \min_{u(\cdot)} J
  = 
  \min_{u(\cdot)} \int_{t_{0}}^{t_{1}} C(x(t), u(t))\,dt,
\end{equation*}
subject to Eq.~\eqref{eq:NonlinearControlSystem}, i.e., $\dot{x} = f(x, u)$, and the endpoint constraints $x(t_{0}) = x_{0}$ and $x(t_{1}) = x_{1}$.

A Hamiltonian structure comes into play with the introduction of the augmented cost functional: Let us introduce the costate $\lambda(t) \in T^{*}M$ and define
\begin{align*}
  \hat{S} &\defeq \int_{t_{0}}^{t_{1}} \brackets{ C(x(t), u(t)) + \ip{ \lambda(t) }{ \dot{x}(t) - f(x(t), u(t)) } } dt
  \\
  &= \int_{t_{0}}^{t_{1}} \brackets{ \ip{ \lambda(t) }{ \dot{x}(t) } - \hat{H}(x(t), \lambda(t), u(t)) } dt
\end{align*}
with the {\em control Hamiltonian}:
\begin{equation*}
  \hat{H}: T^{*}M \oplus E \to \R;
  \quad
  \hat{H}(\lambda_{x} , u_{x}) = \hat{H}(x, \lambda , u) \defeq \ip{\lambda_{x}}{f(u_{x})} - C(u_{x}),
\end{equation*}
where we wrote $\lambda_{x} \defeq (x, \lambda) \in T_{x}^{*}M$ and $u_{x} \defeq (x, u) \in E_{x}$ (recall that $E = M \times \R^{d}$ is a trivial vector bundle over $M$).
If the cost function is invariant under the $G$-action $\Psi$ defined in Eq.~\eqref{eq:Psi}, i.e., for any $g \in G$, 
\begin{equation}
  \label{eq:C-symmetry}
  C \circ \Psi_{g} = C,
\end{equation}
then the control Hamiltonian $\hat{H}$ has a symmetry in the following sense:
Define an action of $G$ on the bundle $T^{*}M \oplus E$ by, for any $g \in G$, 
\begin{equation*}
  \hat{\Psi}_{g}: T^{*}M \oplus E \to T^{*}M \oplus E;
  \quad
  (\lambda_{x} , u_{x}) \mapsto \parentheses{ T^{*}\Phi_{g^{-1}}(\lambda_{x}) , \Psi_{g}(u_{x}) },
\end{equation*}
where $T^{*}\Phi_{g^{-1}}: T^{*}M \to T^{*}M$ is the cotangent lift of $\Phi_{g}$.
Then it is easy to show that the control Hamiltonian $\hat{H}$ is invariant under the $G$-action defined above, i.e.,
\begin{equation}
  \label{eq:hatH-symmetry}
  \hat{H} \circ \hat{\Psi}_{g} = \hat{H}
\end{equation}
for any $g \in G$.

Now, for an arbitrary fixed $\lambda_{x} \in T_{x}^{*}M$, define $\F_{\rm c}\hat{H}(\lambda_{x}, \,\cdot\,): E_{x} \to E_{x}^{*}$ as follows: For any $w_{x} \in E_{x}$, 
\begin{equation*}
  \ip{ \F_{\rm c}\hat{H}(\lambda_{x}, u_{x}) }{ w_{x} } = \left.\od{}{\varepsilon} \hat{H}(\lambda_{x}, u_{x} + \varepsilon\,w_{x}) \right|_{\varepsilon=0},
\end{equation*}
where $\ip{\,\cdot\,}{\,\cdot\,}$ on the left-hand side is the natural pairing between elements in $E_{x}^{*}$ and $E_{x}$.
We assume that the optimal control $u^{\star}_{x}: T_{x}^{*}M \to E_{x} \cong \R^{d}$ is uniquely determined by the equation
\begin{equation*}
   \F_{\rm c}\hat{H}\parentheses{ \lambda_{x}, u^{\star}_{x}(\lambda_{x}) } = 0
\end{equation*}
for any $\lambda_{x} \in T_{x}^{*}M$.
This gives rise to the fiber-preserving bundle map
\begin{equation*}
  u^{\star}: T^{*}M \to E;
  \quad
  \lambda_{x} \mapsto u^{\star}_{x}(\lambda_{x}).
\end{equation*}
Then one may show that the optimal control $u^{\star}: T^{*}M \to E$ is equivariant under the $G$-actions, i.e.,
\begin{equation}
  \label{eq:u^star-symmetry}
  \Psi_{g} \circ u^{\star} = u^{\star} \circ T^{*}\Phi_{g^{-1}},
\end{equation}
and so we may define the optimal Hamiltonian $H: T^{*}M \to \R$ by $H \defeq \hat{H} \circ u^{\star}$, or more explicitly,
\begin{equation}
  \label{eq:H}
  H(\lambda_{x}) \defeq \hat{H}( \lambda_{x} , u^{\star}_{x}(\lambda_{x}) )
  = \ip{ \lambda_{x} }{ f(u^{\star}_{x}(\lambda_{x})) } - C(u^{\star}_{x}(\lambda_{x})).
\end{equation}
Then the symmetry of the control Hamiltonian and the optimal control, i.e., Eq.~\eqref{eq:hatH-symmetry} and \eqref{eq:u^star-symmetry}, imply that of the optimal Hamiltonian $H$, i.e., 
\begin{equation*}
  H \circ T^{*}\Phi_{g^{-1}} = H
\end{equation*}
for any $g \in G$.

The Pontryagin maximum principle says that the optimal flow on $M$ of the control system~\eqref{eq:NonlinearControlSystem} is necessarily the projection to $M$ of the Hamiltonian flow on $T^{*}M$ with the optimal Hamiltonian $H$ defined above.
Specifically, let $\Omega$ be the standard symplectic form on $T^{*}M$, $\pi_{M}: T^{*}M \to M$ the cotangent bundle projection, and $X_{H}$ the Hamiltonian vector field defined by
\begin{equation}
  \label{eq:HamiltonianSystem}
  i_{X_{H}}\Omega = dH;
\end{equation}
then there exists a solution $\lambda: [t_{0},t_{1}] \to T^{*}M$ of the above Hamiltonian system with $\pi_{M}(\lambda(t_{0})) = x_{0}$ and $\pi_{M}(\lambda(t_{1})) = x_{1}$ such that its projection to $M$, $\pi_{M} \circ \lambda: [t_{0},t_{1}] \to M$, is the optimal trajectory of the control system (see, e.g., \citet[Chapter~12]{AgSa2004} for more details).
In other words, the optimal flow on $M$ of the control system is given by the vector field $T\pi_{M}(X_{H})$ on $M$.

\subsection{Poisson Reduction and Hamilton--Poincar\'e Equations}
\label{ssec:PoissonRedAndHPEq}
We saw that the optimal Hamiltonian $H$ is $G$-invariant; this implies that we can apply the results of symmetry reduction of Hamiltonian systems to Eq.~\eqref{eq:HamiltonianSystem} to obtain a reduced Hamiltonian system related to the optimal flow.
Such reduction is helpful in practical applications, since it helps one to reduce the number of unknowns in the Hamiltonian system \eqref{eq:HamiltonianSystem}.

Reduction of Hamiltonian systems is a well-developed subject, whose roots go back to the symplectic reduction of \citet{MaWe1974}; there have been substantial subsequent developments (see \citet{MaMiOrPeRa2007} and references therein).
In our case, the Poisson version of the cotangent bundle reduction (see \citet{CeMaPeRa2003} and \citet[Section~2.3]{MaMiOrPeRa2007}; see also \citet{MoMaRa1984} and \citet{Mo1986}) turns out to be a natural choice for the following reason: Recall that we derived the reduced control system~\eqref{eq:ReducedControlSystem} on the quotient configuration space $M/G$ using the bundle $T(M/G) \oplus \tilde{\mathfrak{g}}$ over $M/G$.
{\em It is natural to expect and also is desirable that the maximum principle, originally formulated on $T^{*}M$, reduces to the dual $T^{*}(M/G) \oplus \tilde{\mathfrak{g}}^{*}$, which is also a bundle over $M/G$; then we may consider it a reduced version of the maximum principle (see the dashed arrow in Fig.~\ref{fig:RelatingOCPs})}.
The Poisson version of the cotangent bundle reduction works precisely this way: The Poisson structure on $T^{*}M$ reduces to that on $T^{*}M/G \cong T^{*}(M/G) \oplus \tilde{\mathfrak{g}}^{*}$; accordingly, Hamilton's equations reduce to the Hamilton--Poincar\'e equations~\cite{CeMaPeRa2003} defined of $T^{*}(M/G) \oplus \tilde{\mathfrak{g}}^{*}$.

As shown in \citet[Lemma~2.3.3 on p.~74]{MaMiOrPeRa2007}, the identification of $T^{*}M$ with $T^{*}(M/G) \oplus \tilde{\mathfrak{g}}^{*}$ is provided by the dual of the inverse of $\alpha_{\mathcal{A}}$ defined in Eq.~\eqref{eq:alpha_mathcalA}:
\begin{equation}
  \label{eq:alpha_mathcalA^-1^star}
  (\alpha_{\mathcal{A}}^{-1})^{*}: T^{*}M/G \to T^{*}(M/G) \oplus \tilde{\mathfrak{g}}^{*};
  \quad
  [\lambda_{x}]_{G} \mapsto \hl_{x}^{*}(\lambda_{x}) \oplus [x, {\bf J}(\lambda_{x})]_{G},
\end{equation}
where $\hl_{x}^{*}: T^{*}_{x}M \to T^{*}_{\bar{x}}(M/G)$ is the adjoint of the horizontal lift $\hl_{x}: T_{\bar{x}}(M/G) \to T_{x}M$ associated with the connection form $\mathcal{A}: TM \to \mathfrak{g}$, and ${\bf J}: T^{*}M \to \mathfrak{g}^{*}$ is the momentum map corresponding to the $G$-symmetry: Let $\xi$ be an arbitrary element in $\mathfrak{g}$ and $\xi_{M} \in \mathfrak{X}(M)$ its infinitesimal generator; then ${\bf J}$ is defined by
\begin{equation}
  \label{eq:J}
  \ip{ {\bf J}(\lambda_{x}) }{ \xi } = \ip{ \lambda_{x} }{ \xi_{M}(x) }.
\end{equation}
Recall from, e.g., \citet[Section~11.4]{MaRa1999} that Noether's theorem says that a $G$-invariance of $H$ implies that ${\bf J}$ is conserved along the flow of the Hamiltonian vector field $X_{H}$.
We note that \citet{Su1995} formulated a generalized version of Noether's theorem for optimal control systems that does not require some of the assumptions we made here; however the original one suffices for our purpose here.

\citet{CeMaPeRa2003} exploit this identification to reduce the Hamiltonian dynamics with a $G$-invariant Hamiltonian $H: T^{*}M \to \R$ as follows:
The $G$-invariance implies that one can define the reduced Hamiltonian on $T^{*}M/G$, which is identified with $T^{*}(M/G) \oplus \tilde{\mathfrak{g}}^{*}$ by Eq.~\eqref{eq:alpha_mathcalA^-1^star}, i.e., one has $\bar{H}: T^{*}(M/G) \oplus \tilde{\mathfrak{g}}^{*} \to \R$.
Then, through the reduction of Hamilton's phase space principle, i.e.,
\begin{equation*}
  \delta \int_{t_{0}}^{t_{1}} \brackets{ \ip{p}{\dot{q}} - H(q,p) } dt = 0
\end{equation*}
with $\delta q(t_{0}) = \delta q(t_{1}) = 0$, one obtains the Hamilton--Poincar\'e equations defined on $T^{*}(M/G) \oplus \tilde{\mathfrak{g}}^{*}$:
\begin{equation}
  \label{eq:Hamilton-Poincare}
  \begin{array}{c}
    \dot{\bar{x}} = \pd{\bar{H}}{\bar{\lambda}},
    \qquad
    \tilde{\xi} = \pd{\bar{H}}{\tilde{\mu}},
    \bigskip\\
    \covd{\bar{\lambda}}{t} = -\pd{\bar{H}}{\bar{x}} - \ip{ \tilde{\mu} }{ i_{\dot{\bar{x}}}\tilde{\mathcal{B}} },
    \qquad
    \covd{\tilde{\mu}}{t} = \ad^{*}_{\tilde{\xi}}\tilde{\mu},
  \end{array}
\end{equation}
where $\bar{\lambda} \oplus \tilde{\mu}$ is an element in $T^{*}(M/G) \oplus \tilde{\mathfrak{g}}^{*}$; $\tcovd{}{t}$ is the covariant derivative in the associated bundle (see \citet[Section~2.3]{CeMaRa2001} and \citet{CeMaPeRa2003}); $\tilde{\mathcal{B}}$ is the reduced curvature form defined as follows (see \citet[Lemma~4.5]{CeMaPeRa2003}):
Let $\hor_{x}: T_{x}M \to T_{x}M$ be the horizontal component defined by the connection form $\mathcal{A}$:
\begin{equation*}
  \hor_{x}(\mathcal{X}_{x}) = \mathcal{X}_{x} - (\mathcal{A}_{x}(\mathcal{X}_{x}))_{M}(x),
\end{equation*}
where $(\,\cdot\,)_{M}: \mathfrak{g} \to \mathfrak{X}(M)$ is the infinitesimal generator.
Also, let $\mathcal{B}$ be the curvature of the connection form $\mathcal{A}: TM \to \mathfrak{g}$, i.e., it is the $\mathfrak{g}$-valued two-form on $M$ defined by
\begin{equation*}
  \mathcal{B}_{x}(\mathcal{X}_{x}, \mathcal{Y}_{x}) = d\mathcal{A}_{x}(\hor_{x}(\mathcal{X}_{x}), \hor_{x}(\mathcal{Y}_{x})).
\end{equation*}
Then the reduced curvature form $\tilde{\mathcal{B}}$ is the $\tilde{\mathfrak{g}}$-valued two-form on $M/G$ defined by
\begin{equation*}
  \tilde{\mathcal{B}}_{\bar{x}}(X_{\bar{x}}, Y_{\bar{x}}) = [x, \mathcal{B}_{x}(\mathcal{X}_{x}, \mathcal{Y}_{x})]_{G}
\end{equation*}
for any $X_{\bar{x}}, Y_{\bar{x}} \in T_{\bar{x}}(M/G)$ and $\mathcal{X}_{x}, \mathcal{Y}_{x} \in T_{x}M$ such that $T_{x}\pi(\mathcal{X}_{x}) = X_{\bar{x}}$ and $T_{x}\pi(\mathcal{Y}_{x}) = Y_{\bar{x}}$.
In coordinates (see \citet[Section~4]{CeMaRa2001b} for details), Eq.~\eqref{eq:Hamilton-Poincare} becomes
\begin{equation*}
  \begin{array}{c}
    \dot{\bar{x}}^{\alpha} = \pd{\bar{H}}{\bar{\lambda}^{\alpha}},
    \qquad
    \tilde{\xi}^{a} = \pd{\bar{H}}{\tilde{\mu}_{a}},
    \bigskip\\
    \dot{\bar{\lambda}}_{\alpha} = -\pd{\bar{H}}{\bar{x}^{\alpha}}
    - \tilde{\mu}_{a} \parentheses{
      \mathcal{B}^{a}_{\beta \alpha} \dot{\bar{x}}^{\beta} + \mathcal{A}^{b}_{\alpha} C^{a}_{d b} \pd{\bar{H}}{\tilde{\mu}_{d}}
    },
    \qquad
    \dot{\tilde{\mu}}_{a} = \tilde{\mu}_{b} C^{b}_{d a} \parentheses{ \pd{\bar{H}}{\tilde{\mu}_{d}} - \mathcal{A}^{d}_{\alpha} \dot{\bar{x}}^{\alpha} },
  \end{array}
\end{equation*}
where $\tilde{\xi}^{a}$ and $\tilde{\mu}_{a}$ are the {\em locked body angular velocity} and its corresponding momentum (see \citet[Section~5.3]{BlKrMaMu1996}) defined by
\begin{equation*}
  \tilde{\xi}^{a} = \xi^{a} + \mathcal{A}^{a}_{\alpha} \dot{\bar{x}}^{\alpha}
  = \parentheses{ \Ad_{g^{-1}}\mathcal{A}_{(\bar{x}, g)}(\dot{\bar{x}}, \dot{g}) }^{a},
  \qquad
  \tilde{\mu}_{a} = \parentheses{ \Ad_{g}^{*} {\bf J}(\lambda_{x}) }_{a}.
\end{equation*}
with $\xi = T_{g}L_{g^{-1}}(\dot{g})$; the coefficients $\mathcal{A}^{a}_{\alpha}$ are defined in the coordinate expression for the connection form $\mathcal{A}$ as follows:
\begin{equation*}
  \mathcal{A}_{(\bar{x}, g)}(\dot{\bar{x}}, \dot{g}) = \Ad_{g}(\xi^{a} + \mathcal{A}^{a}_{\alpha} \dot{\bar{x}}^{\alpha})\, {\bf e}_{a},
\end{equation*}
where $\{ {\bf e}_{a} \}_{a=1}^{\dim G}$ is a basis for the Lie algebra $\mathfrak{g}$.
Also the coefficients $\mathcal{B}^{a}_{\beta\alpha}$ for the curvature are given by
\begin{equation*}
  \mathcal{B}^{a}_{\beta\alpha} = \pd{\mathcal{A}^{a}}{\bar{x}^{\alpha}} - \pd{\mathcal{A}^{a}}{\bar{x}^{\beta}} - C^{a}_{b c} \mathcal{A}^{b}_{\alpha} \mathcal{A}^{c}_{\beta}.
\end{equation*}

\subsection{Poisson Reduction of Pontryagin Maximum Principle}
\label{ssec:PoissonRedOfPMP}
Let us apply the above Poisson reduction to the Hamiltonian system~\eqref{eq:HamiltonianSystem} defined by the maximum principle.
First calculate the reduced optimal Hamiltonian $\bar{H}$ corresponding to the optimal Hamiltonian~\eqref{eq:H}.
Using the identification in Eq.~\eqref{eq:alpha_mathcalA^-1^star} and also the reduced optimal control
\begin{equation*}
  \bar{u}^{\star}: T^{*}M/G \to E/G,
\end{equation*}
which is well-defined due to Eq.~\eqref{eq:u^star-symmetry}, we can rewrite the Hamiltonian $H$ as follows:
\begin{align*}
  H(\lambda_{x}) &= \ip{ (\alpha_{\mathcal{A}}^{-1})^{*}(\lambda_{x}) }{\, \alpha_{\mathcal{A}} \circ f(u^{\star}_{x}(\lambda_{x})) } - C(u^{\star}_{x}(\lambda_{x}))
  \\
  &= \ip{ \hl_{x}^{*}(\lambda_{x}) }{\, \bar{f}_{M/G}^{\star}([\lambda_{x}]_{G}) }
  + \ip{ [x, {\bf J}(\lambda_{x})]_{G} }{\, \bar{f}_{\tilde{\mathfrak{g}}}^{\star}([\lambda_{x}]_{G}) }
  - \bar{C}^{\star}([\lambda_{x}]_{G}),
\end{align*}
where we defined the reduced cost function $\bar{C}: E/G \to \R$ by $\bar{C} \circ \pi^{E}_{G} = C$ and also
\begin{equation*}
  \bar{f}_{M/G}^{\star}: T^{*}M/G \to T(M/G),
  \qquad
  \bar{f}_{\tilde{\mathfrak{g}}}^{\star}: T^{*}M/G \to \tilde{\mathfrak{g}},
  \qquad
  \bar{C}^{\star}: T^{*}M/G \to \R
\end{equation*}
by
\begin{equation*}
  \begin{array}{c}
    \DS \bar{f}_{M/G}^{\star}([\lambda_{x}]_{G}) \defeq \bar{f}_{M/G} \circ \bar{u}^{\star}_{\bar{x}}([\lambda_{x}]_{G}),
    \qquad
    \bar{f}_{\tilde{\mathfrak{g}}}^{\star}([\lambda_{x}]_{G}) \defeq \bar{f}_{\tilde{\mathfrak{g}}} \circ \bar{u}^{\star}_{\bar{x}}([\lambda_{x}]_{G}),
    \medskip\\
    \bar{C}^{\star}([\lambda_{x}]_{G}) \defeq \bar{C} \circ \bar{u}^{\star}_{\bar{x}}([\lambda_{x}]_{G}).
  \end{array}
\end{equation*}
Define the reduced optimal Hamiltonian $\bar{H}: T^{*}(M/G) \oplus \tilde{\mathfrak{g}}^{*} \to \R$ by
\begin{equation}
  \label{eq:barH}
  \bar{H}\parentheses{ \bar{\lambda}_{\bar{x}} \oplus \tilde{\mu}_{\bar{x}} }
  \defeq \ip{ \bar{\lambda}_{\bar{x}} }{ \bar{f}_{M/G}^{\star}\parentheses{ \bar{\lambda}_{\bar{x}} \oplus \tilde{\mu}_{\bar{x}} } }
  + \ip{ \tilde{\mu}_{\bar{x}} }{ \bar{f}_{\tilde{\mathfrak{g}}}^{\star}\parentheses{ \bar{\lambda}_{\bar{x}} \oplus \tilde{\mu}_{\bar{x}} } }
  - \bar{C}^{\star}\parentheses{ \bar{\lambda}_{\bar{x}} \oplus \tilde{\mu}_{\bar{x}} },
\end{equation}
where we identified $T^{*}M/G$ with $T^{*}(M/G) \oplus \tilde{\mathfrak{g}}^{*}$ as the domain of the maps $\bar{f}_{M/G}^{\star}$, $\bar{f}_{\tilde{\mathfrak{g}}}^{\star}$, and $\bar{C}^{\star}$.
Then we have $H(\lambda_{x}) = \bar{H}( \bar{\lambda}_{\bar{x}} \oplus \tilde{\mu}_{\bar{x}} )$ with
\begin{equation*}
  \bar{\lambda}_{\bar{x}} \defeq \hl_{x}^{*}(\lambda_{x}),
  \qquad
  \tilde{\mu}_{\bar{x}} \defeq [x, {\bf J}(\lambda_{x})]_{G}.
\end{equation*}
In coordinates, the reduced optimal Hamiltonian is
\begin{equation*}
  \bar{H}\parentheses{\bar{x}, \bar{\lambda}, \tilde{\mu}} = \bar{\lambda}_{\alpha}\, \bar{f}_{M/G}^{\star,\alpha}\parentheses{\bar{x}, \bar{\lambda}, \tilde{\mu}}
  + \tilde{\mu}_{a}\, \bar{f}_{\tilde{\mathfrak{g}}}^{\star,a}\parentheses{\bar{x}, \bar{\lambda}, \tilde{\mu}}
  - \bar{C}^{\star}\parentheses{\bar{x}, \bar{\lambda}, \tilde{\mu}}.
\end{equation*}

Applying the Hamilton--Poincar\'e equations~\eqref{eq:Hamilton-Poincare} of \citet{CeMaPeRa2003} to this particular choice of $\bar{H}$ gives the following:
\begin{theorem}
  Suppose that the nonlinear control system~\eqref{eq:NonlinearControlSystem} and the cost function have $G$-symmetries in the sense of Eqs.~\eqref{eq:f-symmetry} and \eqref{eq:C-symmetry}.
  Then the necessary condition of the Pontryagin maximum principle reduces to the following set of equations:
  \begin{equation}
    \label{eq:Control_Hamilton-Poincare}
    \begin{array}{c}
      \dot{\bar{x}} = \bar{f}_{M/G}^{\star}\parentheses{ \bar{\lambda}_{\bar{x}} \oplus \tilde{\mu}_{\bar{x}} },
      \qquad
      \tilde{\xi} = \bar{f}_{\tilde{\mathfrak{g}}}^{\star}\parentheses{ \bar{\lambda}_{\bar{x}} \oplus \tilde{\mu}_{\bar{x}} },
      \bigskip\\
      \covd{\bar{\lambda}}{t} = -\pd{\bar{H}}{\bar{x}} - \ip{ \tilde{\mu} }{ i_{\dot{\bar{x}}}\tilde{\mathcal{B}} },
      \qquad
      \covd{\tilde{\mu}}{t} = \ad^{*}_{\tilde{\xi}}\tilde{\mu},
    \end{array}
  \end{equation}
  or, in coordinates,
  \begin{equation}
    \label{eq:Control_Hamilton-Poincare-coordinates}
    \begin{array}{c}
      \dot{\bar{x}}^{\alpha} = \bar{f}_{M/G}^{\star,\alpha}\parentheses{\bar{x}, \bar{\lambda}, \tilde{\mu}},
      \qquad
      \tilde{\xi}^{a} = \bar{f}_{\tilde{\mathfrak{g}}}^{\star,a}\parentheses{\bar{x}, \bar{\lambda}, \tilde{\mu}},
      \bigskip\\
      \dot{\bar{\lambda}}_{\alpha} = -\pd{\bar{H}}{\bar{x}^{\alpha}}
      - \tilde{\mu}_{a} \parentheses{
        \mathcal{B}^{a}_{\beta \alpha} \dot{\bar{x}}^{\beta} + \mathcal{A}^{b}_{\alpha} C^{a}_{d b} \bar{f}_{\tilde{\mathfrak{g}}}^{\star,d}\parentheses{\bar{x}, \bar{\lambda}, \tilde{\mu}}
      },
      \qquad
      \dot{\tilde{\mu}}_{a} = \tilde{\mu}_{b} C^{b}_{d a} \parentheses{ \bar{f}_{\tilde{\mathfrak{g}}}^{\star,d}\parentheses{\bar{x}, \bar{\lambda}, \tilde{\mu}} - \mathcal{A}^{d}_{\alpha} \dot{\bar{x}}^{\alpha} }.
    \end{array}
  \end{equation}
\end{theorem}

\begin{remark}
  Notice that the equations for $(\bar{x}, \bar{\lambda}, \tilde{\mu})$ are decoupled from the second one.
  Thus one first solves this subsystem and then solve the second equation to reconstruct the dynamics in the group variables.
\end{remark}

\begin{remark}
\label{remark:AbelianCase}
  If the Lie group $G$ is Abelian, then the structure constants $C^{a}_{bc}$ vanish, and thus we have
  \begin{equation}
    \label{eq:Control_Hamilton-Poincare-coordinates-Abelian}
    \begin{array}{c}
      \dot{\bar{x}}^{\alpha} = \bar{f}_{M/G}^{\star,\alpha}\parentheses{\bar{x}, \bar{\lambda}, \tilde{\mu}},
      \qquad
      \tilde{\xi}^{a} = \bar{f}_{\tilde{\mathfrak{g}}}^{\star,a}\parentheses{\bar{x}, \bar{\lambda}, \tilde{\mu}},
      \bigskip\\
      \dot{\bar{\lambda}}_{\alpha} = -\pd{\bar{H}}{\bar{x}^{\alpha}}
      - \tilde{\mu}_{a} \mathcal{B}^{a}_{\beta \alpha} \dot{\bar{x}}^{\beta},
      \qquad
      \dot{\tilde{\mu}}_{a} = 0.
    \end{array}
  \end{equation}
  In particular, the last equation gives a conservation of the momentum map ${\bf J}$, which simplifies the set of equations further.
  In the non-Abelian case, the conservation of ${\bf J}$ is ``hidden'' in the last equation of \eqref{eq:Control_Hamilton-Poincare-coordinates} since the new variable $\tilde{\mu}$ is not ${\bf J}$ itself (which is conserved): Recall that we defined $\tilde{\mu}_{a} = \parentheses{ \Ad_{g}^{*} {\bf J}(\lambda_{x}) }_{a}$, which reduces to $\tilde{\mu}_{a} = {\bf J}(\lambda_{x})_{a}$ in the Abelian case.
  Notice also that, after solving for $(\bar{x}, \bar{\lambda})$, the second equation \eqref{eq:Control_Hamilton-Poincare-coordinates-Abelian} is solved by quadrature: The equation reduces to the form $g^{-1}(t)\dot{g}(t) = \zeta(t)$, where $\zeta(t)$ is a known curve in the Lie algebra $\mathfrak{g}$, and thus we can integrate the equation easily to obtain
  \begin{equation*}
    g(t) = \exp\parentheses{ \int_{0}^{t} \zeta(s)\,ds},
  \end{equation*}
  since $G$ is Abelian and thus all the bracket terms in the iterated integrals coming from the Picard iteration vanish (see, e.g., \citet{Is2002}).
\end{remark}

\begin{remark}
  \label{remark:ComparisonWithBlSc2000-2}
  The reduction of \citet{BlSc2000} takes advantage of the fact that the momentum map ${\bf J}$ vanishes due to the transversality condition.
  In our setting, the endpoints are fixed and thus this does not hold in general; hence the equations cannot be reduced to the cotangent bundle $T^{*}(M/G)$ as discussed at the end of Section~3 of \cite{BlSc2000}.
  This is why our reduced equations are slightly more complicated than theirs.
  Note, however, that the vanishing of ${\bf J}$ implies $\tilde{\mu} = 0$ and thus our result simplifies to theirs, except of course the differences in our settings and formulations mentioned in Section~\ref{ssec:MainResultsAndComparison} and Remark~\ref{remark:ComparisonWithBlSc2000}.
\end{remark}

\section{How Do We Choose the Principal Connection?}
\label{sec:PrincipalConnection}
We have shown that an optimal control system with symmetry may be reduced to the Hamilton--Poincar\'e equations~\eqref{eq:Control_Hamilton-Poincare}.
However, we did not address the issue of how we should choose the principal connection form $\mathcal{A}$ introduced in \eqref{eq:mathcalA}.
Whereas sometimes the problem setting provides a natural choice of principal connection, such as the mechanical connection in the falling cat problem~(see, e.g., \citet{Mo1990}), it is often not clear what choice has to be made.
One may realize the same setting as the falling cat problem (the ``purely kinematic'' case discussed below in Example~\ref{ex:PurelyKinematicCase}) by {\em choosing} some particular symmetry subgroup of a larger symmetry group of the system; however, this means that one is forced to make a particular choice of symmetry group even when a larger symmetry group is available.
See, e.g., Examples~\ref{ex:mathcalA-Snakeboard-R2} and \ref{ex:mathcalA-Snakeboard} below: The choice $G = \R^{2}$ realizes the ``purely kinematic'' case but we have $\R^{2} \times SO(2)$ as a larger symmetry group.

In this section, we show a construction of principal construction that does not impose such constraints on the choice of the symmetry group $G$.
This construction is particularly explicit for affine and kinematic control systems (Sections~\ref{ssec:NonholonomicConnection} and \ref{ssec:AffineOptimalControlSystem}), but may as well be formulated for more general settings under certain assumptions (Section~\ref{ssec:MomentumMapAndConnection}).

\subsection{Principal Connection}
Let $\mathcal{O}(x)$ be the orbit of the $G$-action $\Phi$ on $M$ (defined in Section~\ref{ssec:SymmetryInNonlinearControlSystems}) through $x \in M$, and $\mathcal{V}_{x}$ be its tangent space at $x$, i.e.,
\begin{equation*}
  \mathcal{O}(x) \defeq \setdef{ \Phi_{g}(x) \in M }{ g \in G },
  \qquad
  \mathcal{V}_{x} \defeq T_{x}\mathcal{O}(x).
\end{equation*}
Then a {\em principal connection} $\mathcal{H}$ on the principal bundle $\pi: Q \to Q/G$ is given by a $G$-invariant distribution that complements $\mathcal{V}$, i.e.,
\begin{equation*}
  T\Phi_{g}(\mathcal{H}) = \mathcal{H} \text{ for $\forall g \in G$}
  \quad\text{and}\quad
  T_{x}M = \mathcal{H}_{x} \oplus \mathcal{V}_{x} \text{ for $\forall x \in M$}.
\end{equation*}
Then one may find the corresponding {\em principal connection form} $\mathcal{A}: TM \to \mathfrak{g}$ such that $\mathcal{A}_{x}(\xi_{M}(x)) = \xi$ for any $\xi \in \mathfrak{g}$ and $\ker\mathcal{A}_{x} = \mathcal{H}_{x}$ for any $x \in M$.

\subsection{Nonholonomic Connection}
\label{ssec:NonholonomicConnection}
One example of principal connection is the so-called {\em nonholonomic connection} introduced in \citet[Section~6.4]{BlKrMaMu1996} (see also \citet[Section~3]{CeMaRa2001b}) for reduction of nonholonomic mechanical systems.
As we shall see in Section~\ref{ssec:AffineOptimalControlSystem}, the nonholonomic connection---complemented by the results of Section~\ref{ssec:MomentumMapAndConnection}---turns out to be a natural choice of principal connection for {\em affine optimal control systems}.

First we make the following ``dimension assumption''~\cite{BlKrMaMu1996}:
\begin{equation*}
  T_{x}M = \mathcal{D}_{x} + \mathcal{V}_{x},
\end{equation*}
where we recall that $\mathcal{D}$ is the distribution defined by the control vector fields (see Eq.~\eqref{eq:mathcalD}).
Now let (see Fig.~\ref{fig:NonholonomicConnection})
\begin{equation*}
  \mathcal{S}_{x} \defeq \mathcal{D}_{x} \cap \mathcal{V}_{x}.
\end{equation*}
Then one may choose, exploiting an additional geometric structure, a certain complementary subspace $\mathcal{H}_{x}$ of $\mathcal{S}_{x}$ in $\mathcal{D}_{x}$ to write $\mathcal{D}_{x}$ as the direct sum of them:
\begin{equation*}
  \mathcal{D}_{x} = \mathcal{H}_{x} \oplus \mathcal{S}_{x}.
\end{equation*}
One may also introduce a complementary subspace $\mathcal{U}_{x}$ to $\mathcal{S}_{x}$ in $\mathcal{V}_{x}$ as well:
\begin{equation*}
  \mathcal{V}_{x} = \mathcal{S}_{x} \oplus \mathcal{U}_{x}.
\end{equation*}
As a result, we have the following decomposition of the tangent space $T_{x}M$:
\begin{equation*}
  T_{x}M = \mathcal{H}_{x} \oplus \mathcal{V}_{x}
  = \mathcal{H}_{x} \oplus \mathcal{S}_{x} \oplus \mathcal{U}_{x}.
\end{equation*}
If, in addition, $\mathcal{H}$ is $G$-invariant, i.e., $T\Phi_{g}(\mathcal{H}) = \mathcal{H}$, then it defines a principal connection on the principal bundle $\pi: M \to M/G$; it is called the {\em nonholonomic connection}~\cite{BlKrMaMu1996}.
Note, however, that {\em the choice of $\mathcal{H}$ is not unique without some additional structure.}
We will come back to this issue later in the subsection to follow.

\begin{figure}[ht!]
  \centering
  \includegraphics[width=.55\linewidth]{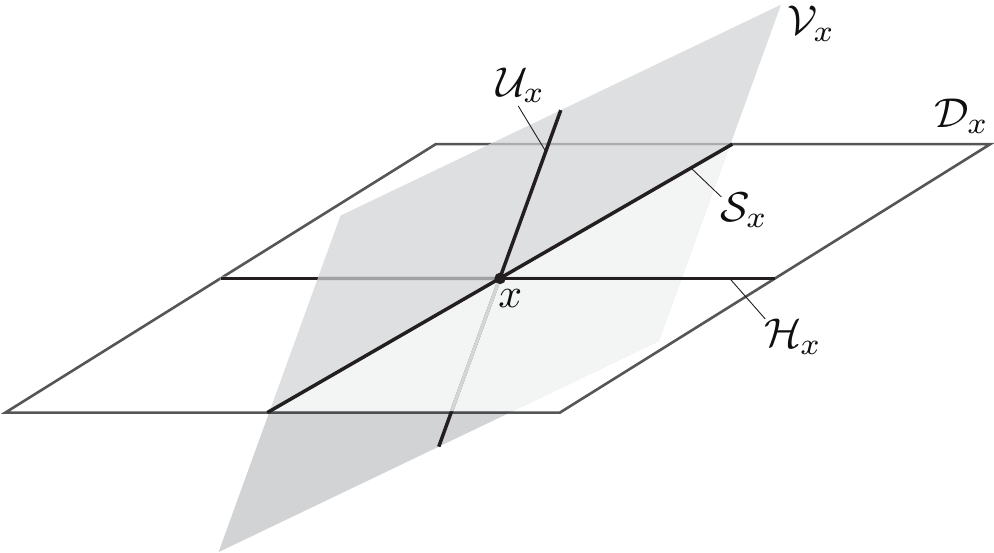}
  \caption{Nonholonomic connection~\cite{BlKrMaMu1996, CeMaRa2001b}. $\mathcal{D}_{x}$ is spanned by the control vector fields $\{X_{i}\}_{i=1}^{d}$; $\mathcal{V}_{x}$ is the tangent space to the group orbit through $x \in M$; $\mathcal{H}_{x}$ defines a principal connection.}
  \label{fig:NonholonomicConnection}
\end{figure}

Using the nonholonomic connection, the reduced control system~\eqref{eq:ReducedControlSystem} can be written as
\begin{equation}
  \label{eq:ReducedControlSystem-affine}
  \dot{\bar{x}} = \bar{X}_{0}(\bar{x}) + \sum_{i=1}^{d} u_{i} \bar{X}_{i}(\bar{x}),
  \qquad
  \tilde{\xi}_{\bar{x}} = \brackets{x, \mathcal{A}_{x} \cdot X_{0}(x) }_{G} + \sum_{i=1}^{d} u_{i}\,\brackets{x, \mathcal{A}_{x} \cdot X_{i}(x) }_{G},
\end{equation}
where $\bar{X}_{i} \defeq T\pi(X_{i})$ for $i = 0, 1, \dots, d$.

The following special case, often called the {\em ``purely kinematic''} case~\cite{BlKrMaMu1996}, gives a simple (although somewhat trivial) example of nonholonomic connection:
\begin{example}[Purely kinematic case---Control of deformable bodies and robotic locomotion]
  \label{ex:PurelyKinematicCase}
  Consider the special case where the tangent space to the group orbit $\mathcal{V}_{x} = T_{x}\mathcal{O}(x)$ exactly complements the $G$-invariant distribution $\mathcal{D}_{x}$, i.e., $\mathcal{S}_{x} = 0$ and thus
  \begin{equation*}
    T_{x}M = \mathcal{D}_{x} \oplus \mathcal{V}_{x}.
  \end{equation*}
  This is the special case called ``purely kinematic'' case or ``Chaplygin systems'' in the context of nonholonomic mechanics~\cite{BlKrMaMu1996}.
  In this case, $\mathcal{D}_{x}$ itself gives the horizontal space $\mathcal{H}_{x}$ and thus defines the connection form $\mathcal{A}: TM \to \mathfrak{g}$ such that $\ker\mathcal{A}_{x} = \mathcal{D}_{x}$ (recall the $G$-symmetry of $\mathcal{D}$, i.e., Eq.~\eqref{eq:mathcalD-symmetry}).
  As a result, Eq.~\eqref{eq:ReducedControlSystem-affine} becomes
  \begin{equation*}
    \dot{\bar{x}} = \bar{X}_{0}(\bar{x}) + \sum_{i=1}^{d} u_{i} \bar{X}_{i}(\bar{x}),
    \qquad
    \tilde{\xi}_{\bar{x}} = \brackets{ x, \mathcal{A}_{x}(X_{0}(x)) }_{G}.
  \end{equation*}
  
  In particular, for the drift-free case, i.e., $X_{0}(x) = 0$, we have $\mathcal{A}_{x}(X_{0}(x)) = 0$ and so $\tilde{\xi}_{\bar{x}} = [x, \mathcal{A}_{x}(\dot{x})]_{G} = 0$, which implies $ \mathcal{A}_{x}(\dot{x}) = 0$.
  With local coordinates $(\bar{x}, g)$ for $M$, we may express the connection $\mathcal{A}$ as
  \begin{equation*}
    \mathcal{A}_{(\bar{x},g)}\parentheses{ \dot{\bar{x}},\dot{g} } = \Ad_{g} \parentheses{ g^{-1}\dot{g} + \mathcal{A}(\bar{x})\,\dot{\bar{x}} },
  \end{equation*}
  where we slightly abused the notation to use $\mathcal{A}(\bar{x})$ as a coordinate expression for the connection form $\mathcal{A}_{(\bar{x},g)}$.
  As a result, Eq.~\eqref{eq:ReducedControlSystem-affine} becomes
  \begin{equation*}
    \dot{\bar{x}} = \sum_{i=1}^{d} u_{i} \bar{X}_{i}(\bar{x}),
    \qquad
    g^{-1}\dot{g} = -\mathcal{A}(\bar{x}) \dot{\bar{x}}.
  \end{equation*}
  This is the basic setting for control of deformable bodies (see, e.g., \citet{Mo1993a}) and also of robotic locomotion (see, e.g., \citet{LiCa1993}, \citet{KeMu1995}, and \citet[Chapters~7 and 8]{MuLiSa1994}); for the former, the connection form $\mathcal{A}$ is defined by the mechanical connection (see, e.g., \citet[Section~2.1]{MaMiOrPeRa2007}) whereas for the latter it is defined by the distribution $\mathcal{D}$ arising from the nonholonomic constraints.
\end{example}

{\em However, in general, $\mathcal{S}_{x} = \mathcal{D}_{x} \cap \mathcal{V}_{x} \neq 0$, and so the choice of the subspace $\mathcal{H}_{x}$ is not trivial, and thus we need to resort to additional ingredients to specify $\mathcal{H}$}; this is the topic of the next subsection.

\subsection{Momentum Map and Principal Connection}
\label{ssec:MomentumMapAndConnection}
To get around the above-mentioned difficulty in specifying the principal connection $\mathcal{H}$, we propose a way to exploit the momentum map (see Eq.~\eqref{eq:J}) corresponding to the symmetry group.
Specifically, we give a generalization of the mechanical connection (see, e.g., \citet[Section~2.1]{MaMiOrPeRa2007}) for systems with degenerate Hamiltonians.

The main result in this subsection, Proposition~\ref{prop:mathcalH-PrincipalConnection}, does not assume the affine optimal control setting, but is proved under quite strong assumptions.
In Section~\ref{ssec:AffineOptimalControlSystem} below, we show that these assumptions are automatically satisfied for a certain class of affine optimal control systems, and also that the construction of principal connection developed here gives a unique choice of the horizontal space $\mathcal{H}$.

Let ${\bf J}: T^{*}M \to \mathfrak{g}^{*}$ be the momentum map for the Hamiltonian system~\eqref{eq:HamiltonianSystem} associated with the optimal control of the nonlinear control system~\eqref{eq:NonlinearControlSystem}.
Recall (see, e.g., \citet[Section~2.1]{MaMiOrPeRa2007}) that the mechanical connection form $\mathcal{A}: TM \to \mathfrak{g}$ is defined by
\begin{equation*}
  \mathcal{A} = \mathbb{I}^{-1} \circ {\bf J} \circ \FL,
\end{equation*}
with the locked inertia tensor $\mathbb{I}: \mathfrak{g} \to \mathfrak{g}^{*}$ and a Lagrangian $L: TM \to \R$; $\FL: TM \to T^{*}M$ is the Legendre transformation defined by
\begin{equation*}
  \ip{ \FL(v_{x}) }{ w_{x} } = \left.\od{}{\varepsilon} L(v_{x} + \varepsilon\,w_{x}) \right|_{\varepsilon=0}
\end{equation*}
for any $v_{x}, w_{x} \in T_{x}M$.
This definition does not directly apply to our setting, since there is usually no such Lagrangian $L$ in the optimal control setting.

Therefore, we need to generalize the notion of the mechanical connection here:
Let $\Phi: G \times M \to M$ be a free and proper action of a Lie group $G$, and $H: T^{*}M \to \R$ be a Hamiltonian.
Define $\FH: T^{*}M \to TM$ by
\begin{equation*}
  \ip{ \FH(\alpha_{x}) }{ \beta_{x} } = \left.\od{}{\varepsilon} H(\alpha_{x} + \varepsilon\,\beta_{x}) \right|_{\varepsilon=0}
\end{equation*}
for any $\alpha_{x}, \beta_{x} \in T^{*}_{x}M$.
We assume that $\FH: T^{*}M \to TM$ is linear and thus $\mathcal{H} \subset TM$ defined by
\begin{equation*}
  \mathcal{H} \defeq \FH \parentheses{ {\bf J}^{-1}(0) }
\end{equation*}
gives a distribution on $M$.
Then, under certain assumptions, $\mathcal{H}$ gives a principal connection on $\pi: M \to M/G$:
\begin{proposition}
  \label{prop:mathcalH-PrincipalConnection}
  Let $\mathcal{V}_{x} = T_{x}\mathcal{O}(x)$ be the tangent space to the group orbit $\mathcal{O}$ of the action $\Phi$.
  Suppose that the Hamiltonian $H$ is $G$-invariant, $\FH: T^{*}M \to TM$ is a linear map that is non-degenerate on ${\bf J}^{-1}(0)$, and also that the intersection of $\mathcal{V}$ and $\mathcal{H}$ is trivial, i.e., $\mathcal{H}_{x} \cap \mathcal{V}_{x} = 0$ for any $x \in M$.
  Then $\mathcal{H}$ defines a principal connection on $\pi: M \to M/G$.
\end{proposition}

\begin{proof}
  See Appendix~\ref{sec:mathcalH-PrincipalConnection-proof}.
\end{proof}

The $G$-invariance of the Hamiltonian is always satisfied in our setting as mentioned in \ref{ssec:PoissonRedAndHPEq}.
The other conditions are somewhat contrived, and it is not clear as to whether one may further scrutinize and weaken the conditions for general settings.
However, in the next subsection, we show that the linearity of $\FH$ and $\mathcal{H}_{x} \cap \mathcal{V}_{x} = 0$ are automatically satisfied for a certain class of affine optimal control problems.

\subsection{Application to Affine Optimal Control Systems with Quadratic Cost Functions}
\label{ssec:AffineOptimalControlSystem}
We apply Proposition~\ref{prop:mathcalH-PrincipalConnection} to a certain class of affine optimal control problems and show that Proposition~\ref{prop:mathcalH-PrincipalConnection} helps us identify the unique principal connection $\mathcal{H}$ even in the non-purely kinematic case.

Consider the following affine optimal control problem:
\begin{equation}
  \label{eq:AffineOptimalControlSystem}
  \dot{x} = X_{0}(x) + \sum_{i=1}^{d} u_{i} X_{i}(x),
  \qquad
  C(x, u) = \frac{1}{2} g_{ij} u_{i} u_{j},
\end{equation}
where $g_{ij} \defeq g(X_{i}, X_{j})$ for $1 \le i, j \le d$ with a $G$-invariant sub-Riemannian metric $g$ on $M$ that is positive-definite on the distribution $\mathcal{D} \defeq \Span\{ X_{1}, \dots, X_{d} \}$.

Let us first introduce a couple of notions to be used in the discussion to follow:
\begin{definition}
  The {\em drift-free} control Hamiltonian $\hat{H}_{\rm df}: T^{*}M \oplus E \to \R$ for the affine control system~\eqref{eq:AffineOptimalControlSystem} is defined by
  \begin{equation*}
    \hat{H}_{\rm df}(\lambda_{x}, u_{x}) \defeq \sum_{i=1}^{d} u_{i} \ip{ \lambda_{x} }{ X_{i}(x) } - \frac{1}{2} g_{ij} u_{i} u_{j}.
  \end{equation*}
  Setting $\F_{\rm c}\hat{H}_{\rm df} = \F_{\rm c}\hat{H} = 0$ gives the optimal control
  \begin{equation}
    \label{eq:u^star-AffineOptimal}
    u^{\star}_{j}(\lambda_{x}) = \ip{\lambda_{x}}{ X_{j}(x) },
  \end{equation}
  and thus we may define the {\em drift-free optimal Hamiltonian} $H_{\rm df}: T^{*}M \to \R$ by
  \begin{equation*}
    H_{\rm df}(\lambda_{x}) \defeq \hat{H}_{\rm df}\parentheses{ \lambda_{x}, u^{\star}_{x}(\lambda_{x}) }
    = \frac{1}{2} g^{ij} \ip{\lambda_{x}}{ X_{i}(x) } \ip{\lambda_{x}}{ X_{j}(x) }.
  \end{equation*}
  For kinematic control systems, i.e., $X_{0}(x) = 0$, we have $\hat{H}_{\rm df} = \hat{H}$ and $H_{\rm df} = H$.
\end{definition}

\begin{remark}
  As we shall see below, the drift-free optimal Hamiltonian is used merely to define a map from ${\bf J}^{-1}(0) \subset T^{*}M$ to $TM$.
  Note also that $H_{\rm df}$ is degenerate unless $d = m \defeq \dim M$, i.e., the system is fully actuated.
\end{remark}

\begin{proposition}
  Suppose that the affine optimal control system~\eqref{eq:AffineOptimalControlSystem} is $G$-invariant in the sense described in Example~\ref{ssec:SymmetryInAffineControlSystems} and Section~\ref{ssec:PMPandSymmetryInOptimalControl}, and also that $\FH_{\rm df}: T^{*}M \to TM$ restricted to ${\bf J}^{-1}(0)$ is non-degenerate.
  Then the distribution
  \begin{equation}
    \label{eq:mathcalH}
    \mathcal{H} \defeq \FH_{\rm df} \parentheses{ {\bf J}^{-1}(0) } \subset TM
  \end{equation}
  defines a principal connection on $\pi: M \to M/G$.
\end{proposition}

\begin{proof}
  Clearly, the $G$-invariance of the optimal control system implies that of $H_{\rm df}$ as well.
  Therefore, by Proposition~\ref{prop:mathcalH-PrincipalConnection}, it remains to show $\mathcal{H}_{x} \cap \mathcal{V}_{x} = 0$.

  First notice that the Legendre transformation $\FH_{\rm df}: T^{*}M \to TM$ is given by
  \begin{equation}
    \label{eq:FH_df}
    \alpha_{x} \mapsto \FH_{\rm df}(\alpha_{x}) \defeq g^{ij} \ip{\alpha_{x}}{ X_{i}(x) } X_{j}(x).
  \end{equation}
  Let $\xi $ be an element in $\mathfrak{g}$ such that $\xi_{M}(x)$ is in $\mathcal{H}_{x}$.
  Then $\xi_{M}(x) = \FH_{\rm df}(\alpha_{x})$ for some $\alpha_{x} \in {\bf J}^{-1}(0)$, and thus, we have, using the definition of the momentum map ${\bf J}$,
  \begin{equation*}
    \ip{ \alpha_{x} }{ \FH_{\rm df}(\alpha_{x}) }
    = \ip{ \alpha_{x} }{ \xi_{M}(x) }
    = \ip{ {\bf J}(\alpha_{x}) }{ \xi }
    = 0.
  \end{equation*}
  On the other hand,
  \begin{equation*}
    \ip{ \alpha_{x} }{ \FH_{\rm df}(\alpha_{x}) }
    = g^{ij} \ip{\alpha_{x}}{ X_{i}(x) } \ip{\alpha_{x}}{ X_{j}(x) }.
  \end{equation*}
  Since $g^{ij}$ is positive definite, we have $\ip{\alpha_{x}}{ X_{j}(x) } = 0$ for $j = 1, \dots d$ and hence $\xi_{M}(x) = \FH_{\rm df}(\alpha_{x}) = 0$.
  Therefore, it follows that $\mathcal{H}_{x} \cap \mathcal{V}_{x} = 0$.
\end{proof}

\begin{remark}
  It is clear from Eqs.~\eqref{eq:mathcalH} and \eqref{eq:FH_df} that $\mathcal{H}_{x}$ is a subspace of $\mathcal{D}_{x}$.
  Since $T_{x}M = \mathcal{H}_{x} \oplus \mathcal{V}_{x}$ as well, the definition of $\mathcal{H}_{x}$ coincides that of Section~\ref{ssec:NonholonomicConnection}.
\end{remark}

Let us first show the purely kinematic case:
\begin{example}[Snakeboard~\cite{OsLeMuBu1994, BlKrMaMu1996, KoMa1997c, BuLe2003a} with $\R^{2}$-symmetry]
  \label{ex:mathcalA-Snakeboard-R2}
  We consider a kinematic optimal control problem of the snakeboard shown in Fig.~\ref{fig:Snakeboard}.
  \begin{figure}[htbp]
    \centering
    \includegraphics[width=.55\linewidth]{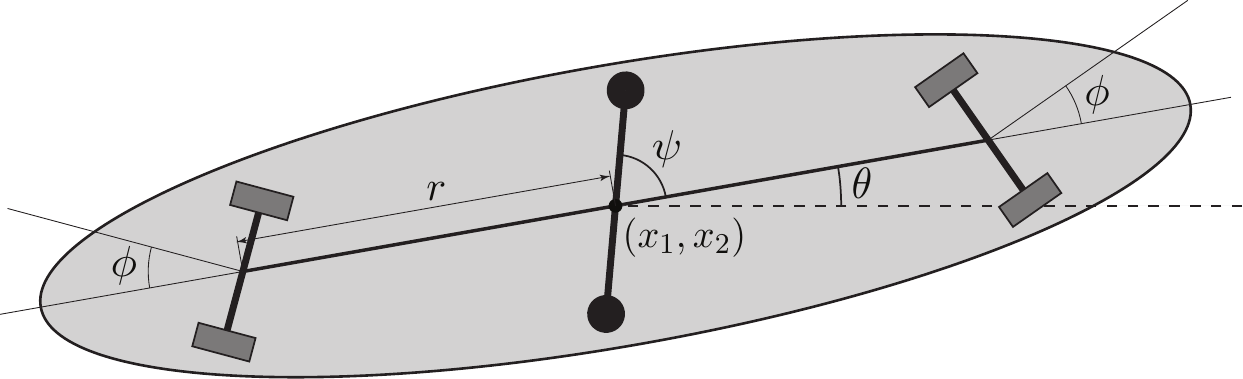}
    \caption{The Snakeboard.}
    \label{fig:Snakeboard}
  \end{figure}
  The configuration space is $M = SE(2) \times \mathbb{S}^{1} \times \mathbb{S}^{1} = \{ (x_{1}, x_{2}, \theta, \psi, \phi) \}$.
  The velocity constraints are given by
  \begin{equation*}
    \dot{x}_{1} + (r \cos\theta \cot\phi)\,\dot{\theta} = 0,
    \quad
    \dot{x}_{2} + (r \sin\theta \cot\phi)\,\dot{\theta} = 0,
  \end{equation*}
  and thus we have $\mathcal{D} = \Span\{ X_{1}, X_{2}, X_{3} \}$ with
  \begin{equation*}
    X_{1}(x) = \cos\theta\,\pd{}{x_{1}} + \sin\theta\,\pd{}{x_{2}} - \frac{\tan\phi}{r}\,\pd{}{\theta},
    \qquad
    X_{2}(x) = \pd{}{\psi},
    \qquad
    X_{3}(x) = \pd{}{\phi},
  \end{equation*}
  where $x = (x_{1}, x_{2}, \theta, \psi, \phi)$.
  Therefore, we may consider the following kinematic control system:
  \begin{equation*}
    \dot{x} = f(x, u) \defeq u_{1} X_{1}(x) + u_{2} X_{2}(x) + u_{3} X_{3}(x),
  \end{equation*}
  or more explicitly,
  \begin{equation*}
    \dot{x}_{1} = u_{1} \cos\theta,
    \qquad
    \dot{x}_{2} = u_{1} \sin\theta,
    \qquad
    \dot{\theta} = -u_{1} \frac{\tan\phi}{r},
    \qquad
    \dot{\psi} = u_{2},
    \qquad
    \dot{\phi} = u_{3}.
  \end{equation*}
  We define the cost function $C: SE(2) \times \R^{3} \to \R$ as follows:
  \begin{equation*}
    C(x, u) = \frac{1}{2}(u_{1}^{2} + u_{2}^{2} + u_{3}^{2}).
  \end{equation*}
  Then the above control system has an $SE(2) \times SO(2)$-symmetry, where $SE(2)$ acting on the $SE(2)$ portion of $M$ by left multiplication and $SO(2)$ acting on the first $\mathbb{S}^{1}$ in $M$, i.e., the variable $\psi$.
  Here we choose the subgroup $G = \R^{2}$ of $SE(2) \times SO(2)$ to show that it realizes the purely kinematic case (see Example~\ref{ex:PurelyKinematicCase}).
  
  Let $\Phi: G \times M \to M$ be the $G$-action on $M$, i.e.,
  \begin{equation*}
    \Phi: ((a, b), (x_{1}, x_{2}, \theta, \psi, \phi)) \mapsto (x_{1} + a, x_{2} + b, \theta, \psi, \phi).
  \end{equation*}
  Also let $\sigma: G \times \R^{3} \to \R^{3}$ be the trivial representation:
  \begin{equation*}
    \sigma: ((a, b), (u_{1}, u_{2}, u_{3})) \mapsto (u_{1}, u_{2}, u_{3}),    
  \end{equation*}
  which induces the action $\Psi: G \times E \to E$ defined by
  \begin{equation*}
    \Psi: ((a, b, \beta), (x_{1}, x_{2}, \theta, \psi, \phi, u_{1}, u_{2}, u_{3})) \mapsto
    (x_{1} + a, x_{2} + b, \theta, \psi, \phi, u_{1}, u_{2}, u_{3}).
  \end{equation*}
   
  The momentum map ${\bf J}: T^{*}M \to T_{(0,0)}\R^{2} \cong \R^{2}$ associated with the action of $\R^{2}$ is
  \begin{equation*}
    {\bf J}(x_{1}, x_{2}, \theta, \psi, \phi, \lambda_{1}, \lambda_{2}, \lambda_{\theta}, \lambda_{\psi}, \lambda_{\phi}) = (\lambda_{1}, \lambda_{2}),
  \end{equation*}
  and then 
  \begin{equation*}
    \mathcal{H} = \FH\parentheses{ {\bf J}^{-1}(0) } = \Span\braces{ X_{1}, X_{2}, X_{3}}.
  \end{equation*}
  So $\mathcal{H} = \mathcal{D}$, and thus this is a purely kinematic case.
\end{example}

A different choice of symmetry group renders the problem non-purely kinematic.
The following example illustrates it; the results here will be later used in the reduction of the system in Example~\ref{ex:Snakeboard}.
\begin{example}[Snakeboard with $\R^{2} \times SO(2)$-symmetry]
  \label{ex:mathcalA-Snakeboard}
  Now we choose $G = \R^{2} \times SO(2)$; this is an Abelian case (see Remark~\ref{remark:AbelianCase}) that gives rise to a non-purely kinematic case.

  Let $\Phi: G \times M \to M$ be the $G$-action on $M$, i.e.,
  \begin{equation*}
    \Phi: ((a, b, \beta), (x_{1}, x_{2}, \theta, \psi, \phi)) \mapsto (x_{1} + a, x_{2} + b, \theta, \psi + \beta, \phi).
  \end{equation*}
  Also let $\sigma: G \times \R^{3} \to \R^{3}$ be the trivial representation:
  \begin{equation*}
    \sigma: ((a, b, \beta), (u_{1}, u_{2}, u_{3})) \mapsto (u_{1}, u_{2}, u_{3}),    
  \end{equation*}
  which induces the action $\Psi: G \times E \to E$ defined by
  \begin{equation*}
    \Psi: ((a, b, \beta), (x_{1}, x_{2}, \theta, \psi, \phi, u_{1}, u_{2}, u_{3})) \mapsto
    (x_{1} + a, x_{2} + b, \theta, \psi + \beta, \phi, u_{1}, u_{2}, u_{3}).
  \end{equation*}
  Then it is straightforward to show that $f$ and $C$ satisfy the symmetry defined in Eqs.~\eqref{eq:f-symmetry} and \eqref{eq:C-symmetry}, respectively.
 
  The momentum map ${\bf J}: T^{*}M \to T_{(0,0)}\R^{2} \times \mathfrak{so}(2) \cong \R^{3}$ associated with the action of $\R^{2} \times SO(2)$ is
  \begin{equation*}
    {\bf J}(x_{1}, x_{2}, \theta, \psi, \phi, \lambda_{1}, \lambda_{2}, \lambda_{\theta}, \lambda_{\psi}, \lambda_{\phi}) = (\lambda_{1}, \lambda_{2}, \lambda_{\psi}),
  \end{equation*}
  and so
  \begin{equation*}
    \mathcal{H} = \FH\parentheses{ {\bf J}^{-1}(0) }
    = \Span\braces{ \cos\theta\,\pd{}{x_{1}} + \sin\theta\,\pd{}{x_{2}} - \frac{\tan\phi}{r}\,\pd{}{\theta},\; \pd{}{\phi} }
    = \Span\braces{ X_{1}, X_{2}}.
  \end{equation*}
  Since $\mathcal{D} = \Span\braces{ X_{1}, X_{2}, X_{3} }$, this is not a purely kinematic case, and $\mathcal{S} = \Span\{ X_{3} \}$.
  The connection form $\mathcal{A}: TM \to \mathfrak{g}$ is then given by
  \begin{equation}
    \label{eq:mathcalA-Snakeboard}
    \mathcal{A}_{(\theta, \phi)} = (dx_{1} + r \cos\theta \cot\phi\, d\theta) \otimes {\bf e}_{1} + (dx_{2} + r \sin\theta \cot\phi\, d\theta) \otimes {\bf e}_{2} + d\psi \otimes {\bf e}_{\psi},
  \end{equation}
  where $\{ {\bf e}_{1}, {\bf e}_{2}, {\bf e}_{\psi} \}$ is a basis for the Lie algebra $T_{(0,0)}\R^{2} \times \mathfrak{so}(2) \cong \R^{3}$.
  We then identify the vertical space $\mathcal{U}$ as follows:
  \begin{equation*}
    \mathcal{U} = \Span\braces{ \pd{}{x_{1}},\; \pd{}{x_{2}},\; \pd{}{\phi} }.
  \end{equation*}
  The reduced curvature form $\tilde{\mathcal{B}}$ is then
  \begin{equation}
    \label{eq:tildemathcalB-Snakeboard}
    \tilde{\mathcal{B}}_{(\theta, \phi)}
    = r \cos\theta \csc^{2}\phi\, d\theta \wedge d\phi \otimes {\bf e}_{1} + r \sin\theta \csc^{2}\phi\, d\theta \wedge d\phi \otimes {\bf e}_{2}.
  \end{equation}
\end{example}

\section{Examples}
\label{sec:Examples}
This section shows various examples to illustrate how the theory specializes to several previous works on the subject (Sections~\ref{ssec:ControlSystemsOnLieGroups}--\ref{ssec:PurelyKinematicCase}), as well as to illustrate how the reduction decouples the optimal control system (Section~\ref{ssec:NonPurelyKinematicCase}).

\subsection{Lie--Poisson Reduction of Optimal Control of Systems on Lie Groups}
\label{ssec:ControlSystemsOnLieGroups}
Consider, as a special case, the nonlinear control system~\eqref{eq:NonlinearControlSystem} on a Lie group $G$, i.e., $M = G$, with symmetry under the action of $G$ on itself by left translation:
\begin{equation*}
  L_{g}: G \to G;
  \quad
  h \mapsto g h
\end{equation*}
for any $g \in G$.
This case is particularly simple because we do not need a principal connection and the reduced system is defined on the Lie algebra $\mathfrak{g}$.

Recall that the associated bundle $E/G = M \times_{G} \R^{d}$ is a bundle over $M/G$; however, $M = G$ here, and so its base space becomes $G/G$, i.e., a point; hence $E/G \cong \R^{d} = \{ \bar{u} \}$ and the map $\bar{f}_{M/G}$ becomes immaterial here.
On the other hand, the quotient $TM/G$ becomes $TG/G \cong \mathfrak{g}$.
Therefore, we have $\bar{f}_{\mathfrak{g}}: \R^{d} \to \mathfrak{g}$ and the control system reduces to
\begin{equation*}
  \xi(t) = \bar{f}_{\mathfrak{g}}(\bar{u}(t)).
\end{equation*}
where $\xi \defeq T_{g}L_{g^{-1}}(\dot{g})$.

In particular, consider the affine control system~\eqref{eq:f-affine} on the Lie group $G$.
The invariance of $X_{0}$, i.e., Eq.~\eqref{eq:X_0-symmetry}, implies that there exists an element $\zeta_{0} \in \mathfrak{g}$ such that $X_{0}(g) = T_{e}L_{g}(\zeta_{0})$ for any $g \in G$, where $e \in G$ is the identity.
Likewise, the invariance of the distribution $\mathcal{D} \subset TG$, i.e., Eq.~\eqref{eq:mathcalD-symmetry}, implies that there exists a subspace $\mathfrak{d}$ in the Lie algebra $\mathfrak{g}$ of $G$ such that $\mathcal{D}_{g} = T_{e}L_{g}(\mathfrak{d})$ for any $g \in G$; so there exists a basis $\{ \zeta_{i} \}_{i=1}^{d}$ for $\mathfrak{d}$ such that $X_{i}(g) = T_{e}L_{g}(\zeta_{i})$ for any $g \in G$ and $i = 1, \dots, d$.
Therefore, Eq.~\eqref{eq:R} implies that the matrix $R(h, g)$ becomes the $d \times d$ identity matrix for any $h, g \in G$.
So the corresponding action $\Psi_{g}: G \times \R^{d} \to G \times \R^{d}$ becomes trivial on the second slot:
\begin{equation}
  \label{eq:Psi-LieGroup}
  \Psi_{g}: (h, u) \mapsto ( g h, u ).
\end{equation}
Hence the quotient $E/G$ becomes
\begin{equation}
  \label{eq:E/G-LieGroup}
  E/G = (G \times \R^{d})/G = (G/G) \times \R^{d} \cong \R^{d} = \{ u \},
\end{equation}
whereas we have $TG/G \cong \mathfrak{g}$.
Now, since $f: G \times \R^{d} \to TG$ takes the form
\begin{equation*}
  f(g, u) = T_{e}L_{g}\parentheses{ \zeta_{0} + \sum_{i=1}^{d} u_{i} \zeta_{i} }, 
\end{equation*}
we obtain the map $\bar{f}_{\mathfrak{g}}: \R^{d} \to \mathfrak{g}$ defined by
\begin{equation}
  \label{eq:barf-LieGroup}
  \bar{f}_{\mathfrak{g}}(u) \defeq \zeta_{0} + \sum_{i=1}^{d} u_{i} \zeta_{i}.
\end{equation}
Therefore, we have the following reduced control system in the Lie algebra $\mathfrak{g}$:
\begin{equation*}
  \xi(t) = \zeta_{0} + \sum_{i=1}^{d} u_{i}(t)\, \zeta_{i},
\end{equation*}
This is the case considered by \citet{Kr1993} (see also \citet[Section~3]{Sa2009}).

Now, assume that the cost function $C: E \to \R$ is also $G$-invariant, i.e., $C \circ \Psi_{h} = C$ for any $h \in G$; then Eq.~\eqref{eq:Psi-LieGroup} implies that, for any $g \in G$, we have $C(g, u) = C(e, u) = \bar{C}(u)$, where $\bar{C}$ is defined on $E/G \cong \R^{d}$ (recall Eq.~\eqref{eq:E/G-LieGroup}).

In this case, the quotient $M/G$ becomes a point and thus the bundle $T(M/G) \oplus \tilde{\mathfrak{g}}$ becomes just $\mathfrak{g}$; as a result, $\tilde{\xi}$ is equal to $\xi$.
Notice also that, since the momentum map is given by ${\bf J}(\lambda_{g}) = T_{e}^{*}R_{g}(\lambda_{g})$, we have
\begin{equation*}
  \tilde{\mu} = [g, {\bf J}(\lambda_{g})]_{G} = [e, \Ad_{g}^{*}{\bf J}(\lambda_{g})]_{G} \cong \Ad_{g}^{*}{\bf J}(\lambda_{g}) = T_{e}^{*}L_{g}(\lambda_{g}) \in \mathfrak{g}^{*},
\end{equation*}
which is the ``body angular momentum.''
Therefore, the Hamilton--Poincar\'e equations~\eqref{eq:Hamilton-Poincare} reduce to the Lie--Poisson equation~\cite{CeMaPeRa2003}:
\begin{equation*}
  \xi = \pd{\bar{H}}{\tilde{\mu}},
  \qquad
  \od{\tilde{\mu}}{t} = \ad^{*}_{\xi}\tilde{\mu}.
\end{equation*}
So Eq.~\eqref{eq:Control_Hamilton-Poincare} becomes
\begin{equation*}
  \xi = \bar{f}_{\tilde{\mathfrak{g}}}^{\star}\parentheses{ \tilde{\mu} },
  \qquad
  \od{\tilde{\mu}}{t} = \ad^{*}_{\xi}\tilde{\mu}.
\end{equation*}
This system with an affine control, Eq.~\eqref{eq:barf-LieGroup}, and the cost function of the form
\begin{equation*}
  C(g, u) = \bar{C}(u) = \frac{1}{2} \sum_{i=1}^{d} I_{i}\,u_{i}^{2}
\end{equation*}
is the case considered by \citet{Kr1993} (see also \citet[Section~5.3]{KoMa1997a} and \citet[Section~7]{Sa2009}).

\subsection{Clebsch Optimal Control Problem}
\label{ssec:ClebschOptimalControl}
Consider the following control system defined by a group action: Let $M$ be a manifold and $G$ a Lie group, and suppose that a $d$-dimensional Lie group $G$ acts on the manifold $M$; hence we have the infinitesimal generator $u_{M} \in \mathfrak{X}(M)$ for any element $u$ in the Lie algebra $\mathfrak{g}$.
Now consider the control system \eqref{eq:NonlinearControlSystem} with $f: M \times \mathfrak{g} \to TM$ defined by
\begin{equation}
  \label{eq:f-Clebsch}
  f(x, u) = u_{M}(x),
\end{equation}
where the element $u$ in $\mathfrak{g}$ is seen as the control here (note that $\mathfrak{g} \cong \R^{d}$ as a vector space).
This is a control system associated with the {\em Clebsch optimal control problem} (see \citet{CoHo2009} and \citet{GaRa2011}).

This problem provides a good example where the action $\sigma: G \times \R^{d} \to \R^{d}$ to the control space $\R^{d}$ is non-trivial (see Remark~\ref{remark:NontrivialActionToControls}).
We define an action of $G$ on $E = M \times \mathfrak{g}$ as follows:
\begin{equation*}
  \Psi_{g}: M \times \mathfrak{g} \to M \times \mathfrak{g};
  \quad
  (x, u) \mapsto \parentheses{ \Phi_{g}(x), \Ad_{g}u }.
\end{equation*}
Then the equivariance of the infinitesimal generator (see, e.g., \citet[Proposition~4.1.26]{AbMa1978}), i.e., $(\Ad_{g}u)_{M}(g x) = T_{x}\Phi_{g}(u_{M}(x))$, gives the equivariance of $f$, i.e., Eq.~\eqref{eq:f-symmetry}.
Now $E/G = M \times_{G} \mathfrak{g} \eqdef \tilde{\mathfrak{g}}$, and so we have $\bar{f}_{M/G}: \tilde{\mathfrak{g}} \to T(M/G)$ and $\bar{f}_{\tilde{\mathfrak{g}}}: \tilde{\mathfrak{g}} \to \tilde{\mathfrak{g}}$ defined by
\begin{equation}
  \label{eq:barf-Clebsch}
  \bar{f}_{M/G}([x, u]_{G}) = T_{x}\pi( u_{M}(x) ) = 0,
  \qquad
  \bar{f}_{\tilde{\mathfrak{g}}}([x, u]_{G}) = [x, \mathcal{A}_{x}( u_{M}(x) ) ]_{G} =  [x, u]_{G}.
\end{equation}
Then the reduced system becomes
\begin{equation*}
  \dot{\bar{x}} = 0,
  \qquad
  \tilde{\xi}_{\bar{x}} = [x, \mathcal{A}_{x}(\dot{x}) ]_{G} = [x, u]_{G}.
\end{equation*}
Hence the point $\bar{x}$ in the base space $M/G$ is fixed, and so the system evolves only in the vertical direction, as one can easily see from Eq.~\eqref{eq:f-Clebsch}.
Therefore, the system is further reduced to
\begin{equation*}
  \mathcal{A}_{x}(\dot{x}) = u.
\end{equation*}

Given a cost function $\ell: M \times \mathfrak{g} \to \mathbb{R}$ such that $\ell(x, u) = \ell(u)$, consider the problem of minimizing the integral
\begin{equation*}
  \int_{0}^{T} \ell(u(t))\,dt
\end{equation*}
subject to Eq.~\eqref{eq:f-Clebsch}, $x(0) = x_{0}$, and $x(T) = x_{T}$; where $x_{0}$ and $x_{T}$ are fixed points in $M$.

It is easy to see that the optimal control is given by
\begin{align}
  \label{eq:u^star-Clebsch}
  \F_{\rm c}\hat{H}\parentheses{ \lambda_{x}, u^{\star}_{x}(\lambda_{x}) } = {\bf J}(\lambda_{x}) - \pd{\ell}{u}(u^{\star}_{x}(\lambda_{x})) = 0
  \iff
  {\bf J}(\lambda_{x}) = \pd{\ell}{u}(u^{\star}_{x}(\lambda_{x})),
\end{align}
assuming this uniquely defines $u^{\star}_{x}(\lambda_{x})$~\cite{GaRa2011}.

Now, from Eq.~\eqref{eq:barf-Clebsch}, $\bar{f}_{M/G}^{\star}([\lambda_{x}]_{G}) = 0$ and $\bar{f}_{\tilde{\mathfrak{g}}}^{\star}([\lambda_{x}]_{G}) = [x, u^{\star}(\lambda_{x})]_{G}$.
Therefore, Eq.~\eqref{eq:Control_Hamilton-Poincare} gives
\begin{equation*}
  \begin{array}{c}
    \dot{\bar{x}} = 0,
    \qquad
    \tilde{\xi} = [x, u^{\star}(\lambda_{x})]_{G},
    \bigskip\\
    \covd{\bar{\lambda}}{t} = -\pd{\bar{H}}{\bar{x}},
    \qquad
    \covd{\tilde{\mu}}{t} = \ad^{*}_{\tilde{\xi}}\tilde{\mu},
  \end{array}
\end{equation*}
The second equation gives 
\begin{equation*}
  [x, \xi]_{G} = [x, u^{\star}(\lambda_{x})]_{G}
  \implies
  \xi = u^{\star}(\lambda_{x}),
\end{equation*}
and the fourth gives, writing $\tilde{\mu} = [x, \mu]_{G}$,
\begin{equation*}
  \covd{}{t}[x, \mu]_{G} = \ad^{*}_{[x, \xi]_{G}} [x, \mu]_{G}
  \implies
  [x, \dot{\mu}]_{G} = [x, \ad^{*}_{\xi}\mu]_{G}
  \implies
  \dot{\mu} = \ad^{*}_{\xi}\mu.
\end{equation*}
since the curve $x(t)$ is vertical, i.e., $\pi(x(t)) = \bar{x}$ is fixed.
However, recall that $\tilde{\mu} = [x, \mu]_{G} \defeq [x, {\bf J}(\lambda_{x})]_{G}$ and thus $\mu = {\bf J}(\lambda_{x})$; then, substituting Eq.~\eqref{eq:u^star-Clebsch} into the above equation, we obtain the following Euler--Poincar\'e equation:
\begin{equation*}
  \od{}{t}\pd{\ell}{u}(u^{\star}_{x}(\lambda_{x})) = \ad^{*}_{u^{\star}(\lambda_{x})}\pd{\ell}{u}(u^{\star}_{x}(\lambda_{x})).
\end{equation*}
This is essentially Theorem~2.2 of \citet{GaRa2011}.

\subsection{Kinematic Optimal Control---Purely Kinematic Case}
\label{ssec:PurelyKinematicCase}
As shown in Sections~\ref{ssec:AffineOptimalControlSystem}, our construction of principal connection is explicit for affine and kinematic sub-Riemannian optimal control problems.
For the purely kinematic case as in Example~\ref{ex:PurelyKinematicCase}, our result recovers that of \cite{Mo1984}:
\begin{example}[Wong's equations~\cite{Wo1970, Mo1984}; {see also \cite[Chapter~4]{CeMaRa2001}}]
  For the kinematic sub-Riemannian optimal control problems~(see \citet{Mo1990, Mo1991a, Mo1993a, Mo1993b, Mo2002} and \citet[Section~7.4]{Bl2003}), we have
  \begin{equation*}
    f(x, u) = \sum_{\alpha=1}^{d} u^{\alpha} X_{\alpha}(x)
  \end{equation*}
  and, given a $G$-invariant sub-Riemannian metric $g$ on $M$ that is positive-definite on the distribution $\mathcal{D} \defeq \Span\{ X_{1}, \dots, X_{d} \}$, the cost function is defined as
  \begin{equation*}
    C(x, u) = \frac{1}{2} g_{\alpha \beta} u^{\alpha} u^{\beta},
  \end{equation*}
  where $g_{\alpha \beta} \defeq g(X_{\alpha}, X_{\beta})$.
  
  Assume that the distribution $\mathcal{D}$ is $G$-invariant and also defines a principal connection form $\mathcal{A}$ on the principal bundle $\pi: M \to M/G$; this is the ``purely kinematic'' case from Section~\ref{ssec:PurelyKinematicCase}.
  In this case, $f(x, u)$ takes values in $\mathcal{D}$; hence $\mathcal{A}(f(x,u)) = 0$ and thus $\bar{f}_{\tilde{\mathfrak{g}}}([x, u]_{G}) = 0$.
  Therefore, Eq.~\eqref{eq:Control_Hamilton-Poincare-coordinates} gives
  \begin{equation*}
    \begin{array}{c}
      \dot{\bar{x}}^{\alpha} = \bar{f}_{M/G}^{\star,\alpha}\parentheses{\bar{x}, \bar{\lambda}, \tilde{\mu}},
      \qquad
      \tilde{\xi}^{a} = 0,
      \bigskip\\
      \dot{\bar{\lambda}}_{\alpha} = -\pd{\bar{H}}{\bar{x}^{\alpha}}
      - \tilde{\mu}_{a} \mathcal{B}^{a}_{\beta \alpha} \dot{\bar{x}}^{\beta},
      \qquad
      \dot{\tilde{\mu}}_{a} = -\tilde{\mu}_{b} C^{b}_{d a} \mathcal{A}^{d}_{\alpha} \dot{\bar{x}}^{\alpha}.
    \end{array}
  \end{equation*}
  Assume that we can write
  \begin{equation*}
    \bar{f}_{M/G}^{\alpha}\parentheses{ \bar{x}, u } = u^{\alpha}.
  \end{equation*}
  Then the optimal control $u^{\star}$ is given by $u^{\star, \alpha} = g^{\alpha \beta} \bar{\lambda}_{\beta}$, and so the reduced optimal Hamiltonian~\eqref{eq:barH} is given by
  \begin{equation*}
    \bar{H}(\bar{x}, \bar{\lambda}) = \frac{1}{2} g^{\alpha \beta} \bar{\lambda}_{\alpha} \bar{\lambda}_{\beta},
  \end{equation*}
  where $g^{\alpha \beta}$ is the inverse of $g_{\alpha \beta}$.
  Therefore, we obtain $\dot{\bar{x}}^{\alpha} = g^{\alpha \beta} \bar{\lambda}_{\beta}$ and $\tilde{\xi}^{a} = 0$ coupled with Wong's equations:
  \begin{equation*}
    \dot{\bar{\lambda}}_{\alpha} = -\frac{1}{2}\pd{g^{\beta \gamma}}{\bar{x}^{\alpha}} \bar{\lambda}_{\beta} \bar{\lambda}_{\gamma}
    - \tilde{\mu}_{a} \mathcal{B}^{a}_{\beta \alpha} \dot{\bar{x}}^{\beta},
    \qquad
    \dot{\tilde{\mu}}_{a} = -\tilde{\mu}_{b} C^{b}_{d a} \mathcal{A}^{d}_{\alpha} \dot{\bar{x}}^{\alpha}.
  \end{equation*}
\end{example}

\subsection{Kinematic Optimal Control---Non-Purely Kinematic Case}
\label{ssec:NonPurelyKinematicCase}
This is the case of main interest in this paper.
Since it is non-purely kinematic, the distribution $\mathcal{D}$ does not define the principal connection, and hence we need to first find the principal connection.
We focus on the Abelian case here, because, as mentioned in Remark~\ref{remark:AbelianCase}, the reduced optimal control system is particularly simple if the symmetry group $G$ is Abelian.
The following kinematic optimal control problem illustrates it (recall that the principal connection is found in Example~\ref{ex:mathcalA-Snakeboard}):
\begin{example}[Snakeboard: Example~\ref{ex:mathcalA-Snakeboard}]
  \label{ex:Snakeboard}
  The optimal control $u^{\star}$, Eq.~\eqref{eq:u^star-AffineOptimal}, is given by
  \begin{equation*}
    u^{\star}_{1} = \lambda_{1} \cos\theta + \lambda_{2} \sin\theta - \lambda_{\theta}\,\frac{\tan\phi}{r},
    \qquad
    u^{\star}_{2} = \lambda_{\psi},
    \qquad
    u^{\star}_{3} = \lambda_{\phi},
  \end{equation*}
  and then the optimal Hamiltonian is
  \begin{equation*}
    H(x, \lambda) = \frac{1}{2}\brackets{
      \parentheses{ \lambda_{1} \cos\theta + \lambda_{2} \sin\theta - \lambda_{\theta}\,\frac{\tan\phi}{r} }^{2}
      + \lambda_{\psi}^{2} + \lambda_{\phi}^{2}
    },
  \end{equation*}
  which gives the optimal control system
  \begin{equation}
    \begin{array}{c}
      \label{eq:OptimalSnakeboard}
      \DS \dot{x}_{1} = \frac{\cos\theta}{r}(r \lambda_{1} \cos\theta + r \lambda_{2} \sin\theta - \lambda_{\theta} \tan\phi),
      \qquad 
      \DS \dot{x}_{2} = \frac{\sin\theta}{r}(r \lambda_{1} \cos\theta + r \lambda_{2} \sin\theta - \lambda_{\theta} \tan\phi),
      \medskip\\
      \DS \dot{\theta} = -\frac{\tan\theta}{r^{2}}(r \lambda_{1} \cos\theta + r \lambda_{2} \sin\theta - \lambda_{\theta} \tan\phi),
      \qquad
      \DS \dot{\psi} = \lambda_{\psi},
      \qquad
      \DS \dot{\phi} = \lambda_{\phi},
      \medskip\\
      \DS \dot{\lambda}_{1} = 0,
      \qquad
      \DS \dot{\lambda}_{2} = 0,
      \qquad
      \DS \dot{\lambda}_{\theta} = \frac{\lambda_{1} \sin\theta - \lambda_{2} \cos\theta}{r}(r \lambda_{1} \cos\theta + r \lambda_{2} \sin\theta - \lambda_{\theta} \tan\phi),
      \medskip\\
      \DS \dot{\lambda}_{\psi} = 0,
      \qquad
      \DS \dot{\lambda}_{\phi} = \frac{\lambda_{\theta} \sec^{2}\phi}{r^{2}}(r \lambda_{1} \cos\theta + r \lambda_{2} \sin\theta - \lambda_{\theta} \tan\phi).
    \end{array}
  \end{equation}

  Let us perform the reduction.
  Introducing $\bar{\lambda} \in T^{*}(M/G)$, $\tilde{\xi} \in \tilde{\mathfrak{g}}$, and $\tilde{\mu} \in \tilde{\mathfrak{g}}^{*}$ defined by (see Eq.~\eqref{eq:mathcalA-Snakeboard} for the expression of the connection form $\mathcal{A}$)
  \begin{equation*}
    \begin{array}{cc}
      \DS
      \bar{\lambda}_{(\theta,\phi)}
      = (\bar{\lambda}_{\theta}, \bar{\lambda}_{\phi})
      \defeq \hl^{*}_{x}(\lambda_{x}) = 
      \parentheses{
        \lambda_{\theta} - \lambda_{1}\, r \cot\phi \cos\theta - \lambda_{2}\, r \cot\phi \sin\theta,\,
        \lambda_{\phi}
      },
      \medskip\\
      \DS
      \tilde{\xi}_{(\theta,\phi)}
      = \parentheses{ \tilde{\xi}_{1}, \tilde{\xi}_{2}, \tilde{\xi}_{\psi} }
      \defeq [x, \mathcal{A}_{x}(\dot{x}) ]_{G}
      = \parentheses{
        \dot{x}_{1} - (r \cot\phi\,\cos\theta)\, \dot{\theta},\,
        \dot{x}_{2} - (r \cot\phi\,\sin\theta)\, \dot{\theta},\,
        \dot{\psi}
      },
      \medskip\\
      \DS
      \tilde{\mu}_{(\theta,\phi)}
      = \parentheses{ \tilde{\mu}_{1}, \tilde{\mu}_{2}, \tilde{\mu}_{\psi} }
      \defeq [x, {\bf J}(\lambda_{x}) ]_{G}
      = \parentheses{ \lambda_{1}, \lambda_{2}, \lambda_{\psi} },
    \end{array}
  \end{equation*}
  the reduced optimal Hamiltonian~\eqref{eq:barH} is written as
  \begin{equation*}
    \bar{H}\parentheses{ \bar{x}, \bar{\lambda}, \tilde{\mu} }
    = \frac{1}{2}\parentheses{
      \frac{ \bar{\lambda}_{\theta}^{2} \tan^{2}\phi }{ r^{2} }
      + \bar{\lambda}_{\phi}^{2}
      + \tilde{\mu}_{\psi}^{2}
    }.
  \end{equation*}
  As a result, the reduced optimal control system~\eqref{eq:Control_Hamilton-Poincare-coordinates-Abelian} gives (see Eq.~\eqref{eq:tildemathcalB-Snakeboard} for the expressions of the curvature $\tilde{\mathcal{B}}$)
  \begin{equation*}
    \begin{array}{cc}
      \DS
      \dot{\theta} = \frac{\tan^{2}\phi}{r^{2}}\,\bar{\lambda}_{\theta},
      \qquad
      \dot{\phi} = \bar{\lambda}_{\phi},
      \qquad
      \tilde{\xi}_{1} = 0,
      \qquad
      \tilde{\xi}_{2} = 0,
      \qquad
      \tilde{\xi}_{\psi} = \tilde{\mu}_{\psi},
      \medskip\\
      \DS
      \dot{\bar{\lambda}}_{\theta}
      = \bar{\lambda}_{\phi}\, r \csc^{2}\phi\, \parentheses{ \tilde{\mu}_{1} \cos\theta + \tilde{\mu}_{2} \sin\theta },
      \qquad
      \dot{\bar{\lambda}}_{\phi}
      = -\bar{\lambda}_{\theta}\, \sec^{2}\phi\, \parentheses{ \bar{\lambda}_{\theta} \tan\phi + \tilde{\mu}_{1}\, r \cos\theta + \tilde{\mu}_{2}\, r \sin\theta },
      \bigskip\\
      \DS
      \dot{\tilde{\mu}}_{1} = 0,
      \qquad
      \dot{\tilde{\mu}}_{2} = 0,
      \qquad
      \dot{\tilde{\mu}}_{\psi} = 0.
    \end{array}
  \end{equation*}
  This system is significantly simpler than the original optimal control system~\eqref{eq:OptimalSnakeboard}: Notice that we now have a decoupled subsystem for the variables $(\theta, \phi, \bar{\lambda}_{\theta}, \bar{\lambda}_{\phi})$; so we may first solve the subsystem and then obtain the dynamics for $(x, y, \psi)$ by quadrature (see Remark~\ref{remark:AbelianCase}).
\end{example}

\section{Conclusion}
We introduced the idea of symmetry reduction and the related geometric tools in Hamiltonian mechanics to nonlinear optimal control systems to define reduced optimal control problems.
Our main focus was on affine and kinematic optimal control problems.
Particularly, we identified a natural choice of principal connection in such problems to perform the reduction explicitly.
The principal connection provides a way to decouple the control system into subsystems, and also, combined with a Poisson reduction to the Pontryagin maximum principle, decouples the corresponding optimal control system into subsystems as well.
The resulting reduced optimal control system is shown to specialize to some previous works.
We also illustrated, through a simple kinematic optimal control problem, how the reduction simplifies the optimal control system.

\section*{Acknowledgments}
I would like to thank the referees, Anthony Bloch, Mar\'ia Barbero-Li\~n\'an, Matthias Kawski, Taeyoung Lee, Melvin Leok, and Joris Vankerschaver for helpful comments and discussions.
This work was partially supported by the National Science Foundation under the grant DMS-1010687.

\appendix

\section{Proof of Proposition~\ref{prop:mathcalH-PrincipalConnection}}
\label{sec:mathcalH-PrincipalConnection-proof}
\begin{lemma}
  If the Hamiltonian $H: T^{*}M \to \R$ is $G$-invariant, then the distribution $\mathcal{H} \subset TM$ is $G$-invariant as well, i.e., $T\Phi_{g}(\mathcal{H}) = \mathcal{H}$ for any $g \in G$.
\end{lemma}

\begin{proof}
  Let us first show that $\FH: T^{*}M \to TM$ is equivariant, i.e., $T\Phi_{g} \circ \FH = \FH \circ T^{*}\Phi_{g^{-1}}$.
  For any $\alpha_{x} \in T^{*}_{x}M$ and $\beta_{g x} \in T^{*}_{g x}M$, we have, using the $G$-invariance of $H$,
  \begin{align*}
    \ip{ \beta_{x} }{ T\Phi_{g} \circ \FH(\alpha_{x}) }
    &= \ip{ T^{*}\Phi_{g}(\beta_{x}) }{ \FH(\alpha_{x}) }
    \\
    &= \left. \od{}{\eps} H \parentheses{ \alpha_{x} + \eps\,T^{*}\Phi_{g}(\beta_{x}) } \right|_{\eps=0}
    \\
    &= \left. \od{}{\eps} H \parentheses{ T^{*}\Phi_{g^{-1}}(\alpha_{x}) + \eps\,\beta_{x} } \right|_{\eps=0}
    \\
    &= \ip{ \beta_{x} }{ \FH \circ T^{*}\Phi_{g^{-1}} (\alpha_{x}) }.
  \end{align*}
  On the other hand, ${\bf J}^{-1}(0) \subset T^{*}M$ is $G$-invariant: Let $\alpha_{x} \in {\bf J}^{-1}(0)$; then, for any $g \in G$ and $\xi \in \mathfrak{g}$,  
  \begin{align*}
    \ip{ {\bf J} \circ T^{*}\Phi_{g^{-1}}(\alpha_{x}) }{ \xi }
    &= \ip{ T^{*}\Phi_{g^{-1}}(\alpha_{x}) }{ \xi_{M}(g x) }
    \\
    &= \ip{ \alpha_{x} }{ T\Phi_{g^{-1}} \cdot \xi_{M}(g x) }
    \\
    &= \ip{ \alpha_{x} }{ (\Ad_{g^{-1}}\xi)_{M}(x) }
    \\
    &= \ip{ {\bf J}(\alpha_{x}) }{ \Ad_{g^{-1}}\xi }
    \\
    &= 0,
  \end{align*}
  which implies $T^{*}\Phi_{g^{-1}}(\alpha_{x}) \in {\bf J}^{-1}(0)$; thus we have $T^{*}\Phi_{g^{-1}}\parentheses{ {\bf J}^{-1}(0) } \subset {\bf J}^{-1}(0)$.
  This in turn implies the other inclusion: For if $\alpha_{x} \in {\bf J}^{-1}(0)$ then
  \begin{equation*}
    \alpha_{x} = T\Phi_{g^{-1}} \circ T\Phi_{g}(\alpha_{x}) \in T\Phi_{g^{-1}}({\bf J}^{-1}(0)),
  \end{equation*}
  because $T\Phi_{g}(\alpha_{x}) \in T^{*}\Phi_{g}\parentheses{ {\bf J}^{-1}(0) } \subset {\bf J}^{-1}(0)$ from what we have just shown.
  As a result, we have $T^{*}\Phi_{g^{-1}}\parentheses{ {\bf J}^{-1}(0) } = {\bf J}^{-1}(0)$, and thus 
  \begin{align*}
    T\Phi_{g}(\mathcal{H})
    &= T\Phi_{g} \circ \FH\parentheses{ {\bf J}^{-1}(0) }
    \\
    &= \FH \circ T^{*}\Phi_{g^{-1}} \parentheses{ {\bf J}^{-1}(0) }
    \\
    &= \FH\parentheses{ {\bf J}^{-1}(0) }
    \\
    &= \mathcal{H}. \qedhere
  \end{align*}
\end{proof}

\begin{proof}[Proof of Proposition~\ref{prop:mathcalH-PrincipalConnection}]
  Since $\Phi$ is a free action, any element in $\mathfrak{g}^{*}$ is a regular value (see, e.g., \citet[Section~1.1]{MaMiOrPeRa2007}).
  Therefore, ${\bf J}^{-1}(0) \cap T^{*}_{x}M$ defines a subspace of $T^{*}_{x}M$ of codimension $\dim G$, because ${\bf J}: T^{*}M \to \mathfrak{g}^{*}$ is linear in the fiber variables of $T^{*}M$.
  Since $\FH$ is assumed to be non-degenerate on ${\bf J}^{-1}(0)$, $\mathcal{H} \defeq \FH \parentheses{ {\bf J}^{-1}(0) }$ defines a subspace of $T_{x}M$ of codimension $\dim G$ for each $x \in M$, whereas $\dim \mathcal{V}_{x} = \dim G$.
  Therefore, the assumption $\mathcal{H}_{x} \cap \mathcal{V}_{x} = 0$ implies $T_{x}M = \mathcal{H}_{x} \oplus \mathcal{V}_{x}$.
  By the above lemma, $\mathcal{H}$ is $G$-invariant, and thus defines a principal connection.
\end{proof}

\bibliography{SymRedOptCtrl}
\bibliographystyle{plainnat}

\end{document}